\documentclass[12pt,a4paper,twoside]{article}

\title{\sc Hodge-Helmholtz Decompositions of Weighted Sobolev Spaces
in Irregular Exterior Domains with Inhomogeneous and Anisotropic Media}
\def\shorttitle{Hodge-Helmholtz Decompositions}
\def\pauthor{Dirk Pauly}

\def\mylabelonoff{off}
\def\allowdisbrk{no}

\usepackage{a4,exscale,ifthen,amsfonts,amssymb,amsmath,amscd,graphicx}
\usepackage[english]{babel}
\usepackage[mathscr]{eucal}

\author{{\sf\pauthor}}
\numberwithin{equation}{section}

\newenvironment{acknow}{{\vspace*{1cm}\noindent\bf Acknowledgements }}{}
\newcommand{\bewboxw}{\mbox{}\hfill $\square$ \\}

\newenvironment{proof}{{\noindent\bf Proof }}{\bewboxw}

\newcommand{\keywords}[1]{{\noindent\bf Key Words }#1}
\newcommand{\amsclass}[1]{{\noindent\bf AMS MSC-Classifications }#1}

\ifthenelse{\equal{\mylabelonoff}{on}}
{\newcommand{\mylabel}[1]{\label{#1}\fbox{{\rm #1}}}}{\newcommand{\mylabel}[1]{\label{#1}\makebox[0mm][]{}}}
\ifthenelse{\equal{\allowdisbrk}{yes}}
{\allowdisplaybreaks}{}

\newcommand{\paper}[7]{\bibitem{#1} #2, `#3', {\it #4}, #5, (#6), #7.}

\newcommand{\book}[6]{\bibitem{#1} #2, {\it #3}, #4, #5, (#6).}

\newcommand{\dissavail}[6]{\bibitem{#1} #2, `#3', {\sf Dissertation}, #4, (#5), available from {\tt #6}.}

\newcommand{\diplavail}[6]{\bibitem{#1} #2, `#3', {\sf Diplomarbeit}, #4, (#5), available from {\tt #6}.}

\usepackage{amsfonts,amssymb,amsmath,amscd}
\usepackage[mathscr]{eucal}

\newcommand{\schluss}{\ifodd\value{page}\newpage\thispagestyle{empty}\makebox[0mm][]{}\color{sehrhell}.\fi\end{document}}

\newcommand{\normabst}{\hspace{-0.4ex}}

\newcommand{\ol}[1]{\overline{#1}}

\newcommand{\qtext}[1]{\quad\text{#1}\quad}
\newcommand{\qqtext}[1]{\qquad\text{#1}\qquad}

\newcommand{\Nh}{\frac{N}{2}}

\newcommand{\me}{{-1}}

\newcommand{\intom}{\int_\Omega}

\newcommand{\om}{{\Omega}}

\newcommand{\omq}{\ol{\om}}

\newcommand{\dom}{{\p\om}}

\newcommand{\ohne}{\setminus}

\newcommand{\eps}{\varepsilon}

\newcommand{\epsd}{\hat{\eps}}

\newcommand{\epsh}{\check{\eps}}

\newcommand{\homg}[3]{{{}^{#3}\mathrm{h}^{#1}_{#2}}}

\newcommand{\bX}{{\mathbb{X}}}
\newcommand{\bY}{{\mathbb{Y}}}

\newcommand{\Esm}{E_{\sigma,m}}

\newcommand{\To}{\longrightarrow}

\newcommand{\equi}{\Leftrightarrow}

\newcommand{\restr}[2]{\left.#1 \right|_{#2}}

\newcommand{\Abb}[5]{\begin{array}{ccccc}#1&:&#2&\longrightarrow&#3\\{}&{}&#4&\longmapsto&#5\end{array}}

\newcommand{\bds}{\begin{displaystyle}}
\newcommand{\eds}{\end{displaystyle}}
\newcommand{\beq}{\begin{equation}}
\newcommand{\eeq}{\end{equation}}
\newcommand{\beqa}{\begin{eqnarray}}
\newcommand{\eeqa}{\end{eqnarray}}
\newcommand{\beqanon}{\begin{eqnarray*}}
\newcommand{\eeqanon}{\end{eqnarray*}}
\newcommand{\bamath}{$$\begin{array}}
\newcommand{\eamath}{\end{array}$$}
\newcommand{\ba}{\begin{array}}
\newcommand{\ea}{\end{array}}
\newcommand{\bc}{\begin{center}}
\newcommand{\ec}{\end{center}}
\newcommand{\bmini}{\begin{minipage}}
\newcommand{\emini}{\end{minipage}}

\newtheorem{lem}{Lemma}[section]
\newtheorem{theo}[lem]{Theorem}
\newtheorem{cor}[lem]{Corollary}

\newtheorem{rem}[lem]{Remark}

\newcommand{\set}[2]{\{#1 \;\text{\bf :}\;  #2\}}

\newcommand{\setb}[2]{\big\{#1 \;\text{\bf :}\;  #2\big\}}

\DeclareMathOperator{\B}{B}

\newcommand{\pb}[2]{\overset{#2}{\B}{}^{#1}}

\newcommand{\bqom}{\B^q(\Omega)}
\newcommand{\bonqom}{\pb{q}{\circ}(\Omega)}
\newcommand{\bqpeom}{\B^{q+1}(\Omega)}

\newcommand{\bonqmeom}{\pb{q-1}{\circ}(\Omega)}

\newcommand{\rz}{\mathbb{R}}

\newcommand{\nz}{\mathbb{N}}
\newcommand{\zz}{\mathbb{Z}}

\newcommand{\cz}{\mathbb{C}}

\newcommand{\nzn}{{\nz_0}}
\newcommand{\rd}{{\rz^3}}
\newcommand{\rN}{{\rz^N}}
\newcommand{\rNon}{{\rN\ohne\{0\}}}

\newcommand{\pI}{\mathbb{I}}

\newcommand{\cIb}{\bar{\cI}}
\newcommand{\cJb}{\bar{\cJ}}
\newcommand{\pcI}[3]{\cI^{#1,#2}_{#3}}

\newcommand{\pcIb}[3]{\cIb^{#1,#2}_{#3}}
\newcommand{\pcJb}[3]{\cJb^{#1,#2}_{#3}}

\newcommand{\pcIbq}[2]{\pcIb{q}{#1}{#2}}
\newcommand{\pcJbq}[2]{\pcJb{q}{#1}{#2}}

\newcommand{\pcIp}[3]{{}_{#3}\overset{+}{\cI}{}^{#1,#2}}
\newcommand{\pcIpqmens}{\,\pcIp{q-1}{0}{s}}

\newcommand{\calB}{{\mathscr{B}}}

\newcommand{\calH}{{\mathscr{H}}}

\newcommand{\calS}{{\mathscr{S}}}

\newcommand{\calR}{{\mathscr{R}}}

\newcommand{\calD}{{\mathscr{D}}}

\DeclareMathOperator{\cI}{{\mathscr{I}}}
\DeclareMathOperator{\cJ}{{\mathscr{J}}}

\newcommand{\dvec}[3]{\begin{bmatrix}#1\\#2\\#3\end{bmatrix}}

\DeclareMathOperator{\p}{\partial}

\DeclareMathOperator{\rhoh}{\check{\rho}}
\DeclareMathOperator{\tauh}{\check{\tau}}

\DeclareMathOperator{\loc}{loc}
\DeclareMathOperator{\vox}{vox}

\DeclareMathOperator{\Eins}{\mathbf{1}}
\DeclareMathOperator{\Lin}{Lin}
\DeclareMathOperator{\id}{Id}

\DeclareMathOperator{\supp}{supp}

\DeclareMathOperator{\rot}{rot}

\DeclareMathOperator{\ROT}{ROT}
\DeclareMathOperator{\pdiv}{div}

\DeclareMathOperator{\DIV}{DIV}

\DeclareMathOperator{\ie}{i}
\DeclareMathOperator{\e}{e}

\DeclareMathOperator{\pd}{d\hspace{-0.4ex}}

\DeclareMathOperator{\X}{X}

\DeclareMathOperator{\curl}{curl}

\DeclareMathOperator{\s}{\mathbf{s}}

\DeclareMathOperator{\ei}{\mathbf{e}}
\DeclareMathOperator{\ci}{\mathbf{c}}

\newcommand{\pcoverset}[2]{\overset{#2}{\mathrm{C}}{}^{#1}}
\newcommand{\pc}[1]{\pcoverset{#1}{}}
\newcommand{\cn}[1]{\pcoverset{#1}{\circ}}
\newcommand{\cu}{\pc{\infty}}
\newcommand{\cun}{\cn{\infty}}

\newcommand{\cunom}{\cun(\Omega)}

\newcommand{\lz}{\mathrm{L}^2}
\newcommand{\lzom}{\lz(\Omega)}
\newcommand{\Lz}[1]{\lz_{#1}}
\newcommand{\Lzom}[1]{\Lz{#1}(\Omega)}
\newcommand{\Lzs}{\Lz{s}}
\newcommand{\Lzsom}{\Lz{s}(\Omega)}

\newcommand{\qc}[3]{\overset{#3}{\mathrm{C}}{}^{#1,#2}}
\newcommand{\cq}[1]{\qc{#1}{q}{}}
\newcommand{\cqn}[1]{\qc{#1}{q}{\circ}}
\newcommand{\cqu}{\cq{\infty}}
\newcommand{\cqun}{\cqn{\infty}}
\newcommand{\qcom}[3]{\qc{#1}{#2}{#3}(\Omega)}

\newcommand{\cqunom}{\cqun(\Omega)}

\newcommand{\cqpen}[1]{\qc{#1}{q+1}{\circ}}

\newcommand{\cqpeun}{\cqpen{\infty}}

\newcommand{\cqpeunom}{\cqpeun(\Omega)}

\newcommand{\qlz}[1]{\mathrm{L}^{2,#1}}
\newcommand{\qlzom}[1]{\qlz{#1}(\Omega)}

\newcommand{\lzqom}{\qlzom{q}}
\newcommand{\qLz}[2]{\qlz{#1}_{#2}}
\newcommand{\qLzom}[2]{\qLz{#1}{#2}(\Omega)}
\newcommand{\Lzq}[1]{\qLz{q}{#1}}
\newcommand{\Lzqom}[1]{\Lzq{#1}(\Omega)}
\newcommand{\qLzs}[1]{\qLz{#1}{s}}
\newcommand{\qLzsom}[1]{\qLzs{#1}(\Omega)}
\newcommand{\qLzt}[1]{\qLz{#1}{t}}

\newcommand{\Lzqs}{\qLzs{q}}
\newcommand{\Lzqsom}{\Lzqs(\Omega)}
\newcommand{\Lzqt}{\qLzt{q}}
\newcommand{\Lzqtom}{\Lzqt(\Omega)}

\newcommand{\Lzqloc}{\Lzq{\loc}}
\newcommand{\Lzqlocom}{\Lzqloc(\Omega)}

\newcommand{\lzqpeom}{\qlzom{q+1}}
\newcommand{\Lzqpe}[1]{\qLz{q+1}{#1}}
\newcommand{\Lzqpeom}[1]{\Lzqpe{#1}(\Omega)}

\newcommand{\qh}[4]{\overset{#4}{\mathrm{H}}{}^{#1,#2}_{#3}}

\newcommand{\qhon}[3]{\qh{#1}{#2}{#3}{\circ}}

\newcommand{\qhom}[4]{\qh{#1}{#2}{#3}{#4}(\Omega)}

\newcommand{\pr}[4]{{}_{#3}\overset{#4}{\mathrm{R}}{}^{#1}_{#2}}

\newcommand{\rqs}{\pr{q}{s}{}{}}

\newcommand{\rqn}[1]{\pr{q}{#1}{0}{}}

\newcommand{\rqsn}{\pr{q}{s}{0}{}}

\newcommand{\ronq}[1]{\pr{q}{#1}{}{\circ}}
\newcommand{\ronqpe}[1]{\pr{q+1}{#1}{}{\circ}}
\newcommand{\ronqs}{\pr{q}{s}{}{\circ}}

\newcommand{\ronqt}{\pr{q}{t}{}{\circ}}

\newcommand{\ronqloc}{\pr{q}{\loc}{}{\circ}}

\newcommand{\ronqn}[1]{\pr{q}{#1}{0}{\circ}}
\newcommand{\ronqpen}[1]{\pr{q+1}{#1}{0}{\circ}}
\newcommand{\ronqsn}{\pr{q}{s}{0}{\circ}}

\newcommand{\ronqtn}{\pr{q}{t}{0}{\circ}}

\newcommand{\ronqlocn}{\pr{q}{\loc}{0}{\circ}}

\newcommand{\ronqpevoxn}{\pr{q+1}{\vox}{0}{\circ}}

\newcommand{\rqsom}{\rqs(\Omega)}

\newcommand{\ronqom}[1]{\ronq{#1}(\Omega)}

\newcommand{\ronqsom}{\ronqs(\Omega)}

\newcommand{\ronqtom}{\ronqt(\Omega)}

\newcommand{\ronqlocom}{\ronqloc(\Omega)}

\newcommand{\ronqnom}[1]{\ronqn{#1}(\Omega)}
\newcommand{\ronqpenom}[1]{\ronqpen{#1}(\Omega)}
\newcommand{\ronqsnom}{\ronqsn(\Omega)}

\newcommand{\ronqtnom}{\ronqtn(\Omega)}

\newcommand{\ronqlocnom}{\ronqlocn(\Omega)}

\newcommand{\ronqpevoxnom}{\ronqpevoxn(\Omega)}
\newcommand{\pR}[4]{{}_{#3}\overset{#4}{\mathbf{R}}{}^{#1}_{#2}}

\newcommand{\Rqs}{\pR{q}{s}{}{}}

\newcommand{\Rqsom}{\Rqs(\Omega)}

\newcommand{\pdi}[4]{{}_{#3}\overset{#4}{\mathrm{D}}{}^{#1}_{#2}}
\newcommand{\pdq}[1]{\pdi{q}{#1}{}{}}
\newcommand{\dqpe}[1]{\pdi{q+1}{#1}{}{}}
\newcommand{\dqs}{\pdi{q}{s}{}{}}

\newcommand{\dqt}{\pdi{q}{t}{}{}}

\newcommand{\dqpevox}{\pdi{q+1}{\vox}{}{}}
\newcommand{\dqn}[1]{\pdi{q}{#1}{0}{}}

\newcommand{\dqsn}{\pdi{q}{s}{0}{}}

\newcommand{\dqtn}{\pdi{q}{t}{0}{}}

\newcommand{\dqlocn}{\pdi{q}{\loc}{0}{}}

\newcommand{\dqom}[1]{\pdq{#1}(\Omega)}

\newcommand{\dqsom}{\dqs(\Omega)}

\newcommand{\dqtom}{\dqt(\Omega)}

\newcommand{\dqpevoxom}{\dqpevox(\Omega)}
\newcommand{\dqnom}[1]{\dqn{#1}(\Omega)}

\newcommand{\dqsnom}{\dqsn(\Omega)}

\newcommand{\dqtnom}{\dqtn(\Omega)}

\newcommand{\dqlocnom}{\dqlocn(\Omega)}

\newcommand{\pDi}[4]{{}_{#3}\overset{#4}{\mathbf{D}}{}^{#1}_{#2}}

\newcommand{\Dqs}{\pDi{q}{s}{}{}}

\newcommand{\Dqsom}{\Dqs(\Omega)}

\newcommand{\dH}[3]{{}_{#3}\calH^{#1}_{#2}}
\newcommand{\dhqom}{\dH{q}{}{}(\Omega)}

\newcommand{\nH}[3]{{}_{#3}\tilde{\calH}{}^{#1}_{#2}}

\newcommand{\dhqepsom}[1]{\dH{q}{#1}{\eps}(\Omega)}

\newcommand{\turmd}[5]{{{}^{#5}D^{#1,#2}_{#3,#4}}}
\newcommand{\turmr}[5]{{{}^{#5}R^{#1,#2}_{#3,#4}}}

\newcommand{\skp}[2]{\langle#1,#2\rangle}

\newcommand{\norm}[1]{|\normabst|#1|\normabst|}

\renewcommand{\Esm}{E^+_{\sigma,m}}

\DeclareMathOperator{\grad}{grad}

\newcommand{\gradon}{\overset{\circ}{\grad}}
\newcommand{\curln}{\curl_0}
\newcommand{\curlon}{\overset{\circ}{\curl}}
\newcommand{\curlonn}{\overset{\circ}{\curl}_0}
\newcommand{\pdivn}{\pdiv_0}
\newcommand{\pdivon}{\overset{\circ}{\pdiv}}
\newcommand{\pdivonn}{\overset{\circ}{\pdiv}_0}
\DeclareMathOperator{\gradf}{\bf grad}
\DeclareMathOperator{\curlf}{\bf curl}
\DeclareMathOperator{\pdivf}{\bf div}
\newcommand{\gradfs}{\gradf_S}
\newcommand{\curlfs}{\curlf_S}
\newcommand{\pdivfs}{\pdivf_S}
\newcommand{\Deltaf}{\mbox{\boldmath$\Delta$}}
\newcommand{\Deltafs}{\Deltaf_S}

\newcommand{\sr}[4]{{}_{#3}{\overset{#4}{\underset{\widetilde{\hspace{1mm}}}{\mathrm{R}}}}{}^{#1}_{#2}}
\newcommand{\sd}[4]{{}_{#3}{\overset{#4}{\underset{\widetilde{\hspace{1mm}}}{\mathrm{D}}}}{}^{#1}_{#2}}
\newcommand{\dqss}{\sd{q}{s}{}{}}
\newcommand{\ronqss}{\sr{q}{s}{}{\circ}}
\newcommand{\dqssom}{\dqss(\om)}
\newcommand{\ronqssom}{\ronqss(\om)}

\newcommand{\bLom}[4]{{}_{#4}\mathbb{L}^{#1,#2}_{#3}(\Omega)}
\newcommand{\bLzom}[3]{\bLom{2}{#1}{#2}{#3}}
\newcommand{\bLzqom}[2]{\bLzom{q}{#1}{#2}}
\newcommand{\bLzqsom}[1]{\bLzqom{s}{#1}}
\newcommand{\bLzqsepsom}{\bLzqsom{\eps}}

\newcommand{\clbL}[3]{{}_{#3}\mathbb{L}^{#1}_{#2}}
\newcommand{\clbLz}[2]{\clbL{2}{#1}{#2}}
\newcommand{\clbLzs}[1]{\clbLz{s}{#1}}
\newcommand{\clbLzseps}{\clbLzs{\eps}}
\newcommand{\clbLt}[3]{{}_{#3}\tilde{\mathbb{L}}{}^{#1}_{#2}}
\newcommand{\clbLtz}[2]{\clbLt{2}{#1}{#2}}
\newcommand{\clbLtzs}[1]{\clbLtz{s}{#1}}
\newcommand{\clbLtzseps}{\clbLtzs{\eps}}

\newcommand{\rqme}[1]{\pr{q-1}{#1}{}{}}

\newcommand{\ronqme}[1]{\pr{q-1}{#1}{}{\circ}}

\newcommand{\dqme}[1]{\pdi{q-1}{#1}{}{}}
\newcommand{\dqmen}[1]{\pdi{q-1}{#1}{0}{}}

\newcommand{\classH}{\mathrm{H}}
\newcommand{\classbbH}{\mathbb{H}}
\newcommand{\clH}[1]{\classH_{#1}}
\newcommand{\clbH}[1]{\classbbH_{#1}}
\newcommand{\clbHt}[1]{\tilde{\classbbH}_{#1}}
\newcommand{\clHon}[1]{\overset{\circ}{\classH}{}_{#1}}

\newcommand{\calBon}[1]{\overset{\circ}{\calB}{}^{#1}}

\newcommand{\bD}[3]{{}_{#3}\mathbb{D}^{#1}_{#2}}
\newcommand{\bR}[3]{{}_{#3}\overset{\circ}{\mathbb{R}}{}^{#1}_{#2}}

\newcommand{\bDom}[3]{{}_{#3}\mathbb{D}^{#1}_{#2}(\Omega)}
\newcommand{\bRom}[3]{{}_{#3}\overset{\circ}{\mathbb{R}}{}^{#1}_{#2}(\Omega)}
\newcommand{\bDqsom}{\bD{q}{s}{0}(\Omega)}
\newcommand{\bRqsom}{\bR{q}{s}{0}(\Omega)}

\newcommand{\bXqs}{\bX^q_s}
\newcommand{\bYqs}{\bY^q_s}

\newcommand{\bXqsom}{\bXqs(\Omega)}
\newcommand{\bYqsom}{\bYqs(\Omega)}

\newcommand{\towerH}[2]{\bar{\mathcal{H}}^{#1}_{#2}}
\newcommand{\towerHqs}{\towerH{q}{s}}
\newcommand{\towerP}[2]{\bar{\mathscr{P}}^{#1}_{#2}}
\newcommand{\towerPqsmz}{\towerP{q}{s-2}}
\newcommand{\towerD}[2]{\bar{\mathscr{D}}^{#1}_{#2}}

\newcommand{\towerR}[2]{\bar{\mathscr{R}}^{#1}_{#2}}

\newcommand{\towerHh}[2]{\check{\mathcal{H}}^{#1}_{#2}}
\newcommand{\towerPh}[2]{\check{\mathscr{P}}^{#1}_{#2}}
\newcommand{\calV}{\mathscr{V}}
\newcommand{\calU}{\mathscr{U}}
\newcommand{\towerV}[2]{\bar{\calV}^{#1}_{#2}}
\newcommand{\towerU}[2]{\bar{\calU}^{#1}_{#2}}
\newcommand{\towerUh}[2]{\check{\calU}^{#1}_{#2}}

\DeclareMathOperator{\codim}{codim}

\usepackage{algorithm,algorithmic}

\newcommand{\paulydissregtheoaussenzwei}{\cite[Satz 3.7]{paulydiss}}
\newcommand{\paulydissregtheoaussenlokal}{\cite[Korollar 3.8]{paulydiss}}
\newcommand{\paulydisstopiso}{\cite[Satz 6.10]{paulydiss}}
\newcommand{\paulydissEsmdef}{\cite[Lemma 7.11]{paulydiss}}

\newcommand{\paulytimeharmetadef}{\cite[(2.1), (2.2), (2.3)]{paulytimeharm}}
\newcommand{\paulytimeharmmlcpdef}{\cite[Definition 2.4, Remark 2.5]{paulytimeharm}}

\newcommand{\paulytimeharmdefeps}{\cite[Definition 2.1 and 2.2]{paulytimeharm}}
\newcommand{\paulytimeharmlzzerl}{\cite[(2.7)]{paulytimeharm}}

\newcommand{\paulystaticintdirichlet}{\cite[Lemma 3.8]{paulystatic}}
\newcommand{\paulystaticsecgenstatic}{\cite[section 4]{paulystatic}}
\newcommand{\paulystaticdomromequi}{\cite[Corollary 4.4]{paulystatic}}
\newcommand{\paulystaticsectower}{\cite[section 2]{paulystatic}}
\newcommand{\paulystaticsecstaticoperators}{\cite[section 3]{paulystatic}}
\newcommand{\paulystaticsphdeco}{\cite[Theorem 2.6]{paulystatic}}
\newcommand{\paulystaticremint}{\cite[Remark 2.5]{paulystatic}}
\newcommand{\paulystaticrotdivtheospez}{\cite[Theorem 5.1]{paulystatic}}
\newcommand{\paulystaticLloes}{\cite[Theorem 5.8]{paulystatic}}
\newcommand{\paulystaticoddtower}{\cite[Remark 2.4]{paulystatic}}
\newcommand{\paulystaticremsph}{\cite[Remark 2.3]{paulystatic}}
\newcommand{\paulystaticrotdivcorlem}{\cite[Corollary 3.13, Lemma 3.14]{paulystatic}}
\newcommand{\paulystaticproofdiv}{\cite[Lemma 3.5, Lemma 3.12]{paulystatic}}

\begin{document}

\date{2006}
\maketitle{}

\begin{abstract}
We study in detail Hodge-Helmholtz decompositions in nonsmooth exterior domains $\om\subset\rN$
filled with inhomogeneous and anisotropic media. We show decompositions of alternating differential forms
of rank $q$ belonging to the weighted $\lz$-space $\Lzqsom$\,, $s\in\rz$\,, into irrotational and solenoidal
$q$-forms. These decompositions are essential tools, for example, in electro-magnetic theory for exterior domains.
To the best of our knowledge these decompositions in exterior domains with nonsmooth boundaries
and inhomogeneous and anisotropic media are fully new results.
In the appendix we translate our results to the classical framework of vector analysis $N=3$ and $q=1,2$\,.\\
\\
\amsclass{35Q60, 58A10, 58A14, 78A25, 78A30}
\end{abstract}

\tableofcontents

\section{Introduction}

Hodge-Helmholtz decompositions of square integrable fields, i.e. decompositions in
irrotational and solenoidal fields, are important and strong tools for solving partial differential equations,
for instance, in electro-magnetic theory.

Since formally $\ie\grad$ and $\ie\pdiv$ resp. $\curl$ and $\curl$ are adjoint to each other
and $\curl\grad=0$ and $\pdiv\curl=0$ hold as well, the $\eps$-$\lz$-orthogonal decompositions
\begin{align}
\begin{split}
\lzom&=\clbH{}(\curlonn,\om)\oplus_\eps\eps^\me\clbH{}(\pdivn,\om)\oplus_\eps\dH{}{}{\eps}{}(\om)\qquad,\\
\lzom&=\clbHt{}(\curln,\om)\oplus_\eps\eps^\me\clbHt{}(\pdivonn,\om)\oplus_\eps\nH{}{}{\eps}{}(\om)\qquad,
\end{split}\mylabel{cldeco}
\end{align}
where $\om\subset\rd$ is a domain, are easy consequences of the projection theorem in Hilbert space.
Here $\eps:\om\to\rz^{3\times3}$ is a real valued, symmetric and uniformly bounded and positive definite matrix,
which models material properties, such as the dielectricity or the permeability of the medium, and
$\dH{}{}{\eps}{}(\om)$ resp. $\nH{}{}{\eps}{}(\om)$ denotes the space of Dirichlet resp. Neumann fields.
(See section \ref{clHdef} for the exact definitions of all these spaces.)

This problem may be generalized if we formulate Maxwell's equations in the framework of alternating
differential forms of order $q$\,, short $q$-forms, on some $N$-dimensional Riemannian manifold $\om$\,.
Additionally to the generality and the easy and short notation this approach provides also
a deeper insight into the structure of the underlying problems.
It has become customary following Hermann Weyl \cite{weyl} to denote the exterior derivative
$\pd$ by $\rot$ and the co-derivative $\delta$ by $\pdiv$\,. We will use this notation throughout
this paper and thus we have on $q$-forms
$$\pdiv=(-1)^{(q+1)N}*\rot*\qquad,$$
where $*$ is Hodge's star operator. Since $\rot$ and $\pdiv$ are formally skew adjoint to each other
as well as $\rot\rot=0$ and $\pdiv\pdiv=0$ hold,
the corresponding Hodge-Helmholtz decompositions of $\lz$-forms
\beq\Lzqom{}=\bRom{q}{}{0}\oplus_\eps\eps^\me\bDom{q}{}{0}\oplus_\eps\dhqepsom{}\qquad,\mylabel{qdecon}\eeq
again are easy consequences of the projection theorem.
Here $\eps$ maps $\om$ to the real, linear, symmetric and uniformly bounded and positive definite transformations on $q$-forms.
Furthermore, we denote by $\oplus_\eps$ the orthogonal sum with respect to the $\skp{\eps\,\cdot\,}{\,\cdot\,}_{\lzqom}$-scalar product.
(See section \ref{defsection} for definitions.) For $N=3$ and $q=1$ or $q=2$ we obtain the two
classical decompositions \eqref{cldeco}.

In the case of unbounded domains it is often necessary and useful to work with weighted
Sobolev spaces. Especially in our efforts to completely determine the low frequency asymptotics
of the solutions of the time-harmonic Maxwell equations in exterior domains \cite{paulytimeharm,paulystatic}
and a forthcoming third paper \cite{paulyasym} it has turned out
that decompositions of weighted $\lz$-spaces are necessary and essential tools.

Hence motivated by this in the present paper we want to answer the question,
in which way the weighted $\lz$-space of $q$-forms
$$\Lzqsom:=\setb{F\in\Lzqlocom}{\rho^sF\in\lzqom}\qqtext{,}s\in\rz\qquad,$$
where $\om\subset\rN$ is an exterior domain, i.e. a connected open set with compact complement,
and $\rho:=(1+r^2)^{1/2}$ with $r(x):=|x|$ for $x\in\rN$ denotes a weight function,
may be decomposed into irrotational and solenoidal forms, i.e. $q$-forms with vanishing
rotation $\rot$ resp. divergence $\pdiv$\,.

For the special case $s=0$ Picard has shown \eqref{qdecon} in \cite{decomposition}, \cite{potential}
and (in the classical framework) in \cite{boundaryelectro}, \cite{milani}.
Moreover, for domains $\om$ possessing the `Maxwell local compactness property' {\sf MLCP}
(See section \ref{defsection}.) he proved the representations
\begin{align*}
\bRom{q}{}{0}&=\rot\ronqme{-1}(\om)=\rot\big(\ronqme{-1}(\om)\cap\dqmen{-1}(\om)\big)\qquad,\\
\bDom{q}{}{0}&=\pdiv\dqpe{-1}(\om)=\pdiv\big(\dqpe{-1}(\om)\cap\ronqpen{-1}(\om)\big)\qquad,
\end{align*}
i.e. any form from $\bRom{q}{}{0}$ may be represented as a rotation of a solenoidal form
and any form from $\bDom{q}{}{0}$ may be represented as a divergence of a irrotational form.

Now one may expect for arbitrary $s\in\rz$ the direct decomposition
$$\Lzqsom=\bRqsom\dotplus\eps^\me\bDqsom\dotplus\dhqepsom{s}\qquad.$$
But, as we will see, this holds only for $s$ `near' zero,
since for small $s$ we lose the directness of the decomposition
and for large $s$ the right hand side is too small.
However both negative effects are of finite dimensional nature.

For general $s\in\rz\ohne\pI$ introducing the countable discrete set of (bad) weights
$$\pI=\set{N/2+n}{n\in\nzn}\cup\set{1-N/2-n}{n\in\nzn}$$
Weck and Witsch showed in \cite{sphharm} for the special case $\om=\rN$ and $\eps=\id$\,,
where no Dirichlet forms and no boundary exist, the decompositions
\begin{align*}
\Lzqs=\begin{cases}
\rqsn+\dqsn&,\quad s\in(-\infty,-N/2)\\
\rqsn\dotplus\dqsn&,\quad s\in(-N/2,N/2)\\
\rqsn\dotplus\dqsn\dotplus\calS^q_s&,\quad s\in(N/2,\infty)
\end{cases}
\end{align*}
and the representations
$$\rqsn=\rot\rqme{s-1}\qqtext{,}\dqsn=\pdiv\dqpe{s-1}\qquad.$$
Thereby $\calS^q_s$ is a finite dimensional subspace of $\cqun\big(\rN\ohne\{0\}\big)$
generated by the action of the commutator of the Laplacian and a cut-off function $\eta$\,,
which vanishes near the origin and equals one near infinity,
on the linear hull of some finitely many decaying potential forms in $\rN\ohne\{0\}$\,,
i.e. generalized spherical harmonics multiplied by a negative power of $r$ solving Laplace's equation.
(Here we omit the dependence on the domain $\rN$ and denote the direct sum by $\dotplus$\,.)
We note that Weck and Witsch in \cite{sphharm} even decomposed the Lebesgue-Banach spaces $\text{\rm L}^{p,q}_s$ with $p\in(1,\infty)$ instead of $p=2$\,.
The proof of their results uses heavily the corresponding results for the scalar Laplacian in $\rN$ developed by McOwen in \cite{mcowen}.
For the Hilbert space case $p=2$ these results have been generalized to smooth
(at least $\pc{3}$) exterior domains $\om\subset\rN$ by Bauer
in \cite{bauerdipl}. Unfortunately by their second order approach these techniques
can not be applied to handle inhomogeneities $\eps$ and the smoothness of $\om$ is essential as well.

Results in the classical case $q=N-1$ have been given by Specovius-Neugebauer \cite{specovius}
for $\eps=\id$ and a smooth ($\pc{2}$) exterior domain $\om\subset\rN$\,, $N\geq3$\,.
She considered only this special case and additionally only a weaker version of \eqref{qdecon}\,,
which reads as $\qLzom{N-1}{}=\pr{N-1}{}{0}{\circ}(\om)\oplus\pdiv\pdi{N}{-1}{}{}(\om)$
resp. in the classical language
$$\lzom=\grad\clH{-1}(\grad,\om)\oplus\clH{}(\pdivonn,\om)\qquad.$$
She was able to show for $s\in\rz\ohne\pI$
\begin{align*}
\Lzsom=\begin{cases}
\grad\clH{s-1}(\grad,\om)+\clH{s}(\pdivonn,\om)&,\quad s\in(-\infty,-N/2)\\
\grad\clH{s-1}(\grad,\om)\dotplus\clH{s}(\pdivonn,\om)&,\quad s\in(-N/2,N/2)\\
\grad\clH{s-1}(\grad,\om)\dotplus\clH{s}(\pdivonn,\om)\dotplus\calS_s&,\quad s\in(N/2,\infty)
\end{cases}\qquad,
\end{align*}
where $\calS_s$ corresponds to $\calS^{N-1}_s$\,. We note that she proved the corresponding decompositions even for
the Banach space $\text{\rm L}^p_s(\om)$ with $1<p<\infty$\,. Since she used heavily trace operators and
convolution techniques, her results can not be generalized to nonsmooth boundaries or inhomogeneities $\eps$\,.
Moreover, she showed no further decomposition of $\clH{s}(\pdivonn,\om)$ into Neumann fields and images of $\curl$-terms
(for $N=3$), which is highly important in electro-magnetic theory.

Our main focus is to treat nonsmooth boundaries, i.e. Lipschitz boundaries or even weaker assumptions,
and most of all nonsmooth inhomogeneities corresponding to inhomogeneous and anisotropic media, which are only
asymptotically homogeneous. To the best of our knowledge it was an open question, if those weighted $\lz$-decompositions
hold for inhomogeneous and anisotropic media or for nonsmooth boundaries.
We will allow our transformations $\eps$ to be $\text{\rm L}^\infty$-perturbations of the identity,
i.e. $\eps=\id+\epsd$\,, where $\epsd$ does not need to be compactly supported but decays at infinity.
Moreover, $\epsd$ is not assumed to be smooth. We only require $\epsd\in\pc{1}$
in the outside of an arbitrarily large ball.
Omitting some details for this introductory remarks we will show essentially
for small $s\in(-\infty,-N/2)\ohne\tilde{\pI}$
$$\Lzqsom=\ronqsnom+\eps^\me\dqsnom$$
and for large $s\in(-N/2,\infty)\ohne\tilde{\pI}$
\begin{align*}
\Lzqsom\cap\dhqepsom{}^{\bot_\eps}&=
\begin{cases}
\bRqsom\dotplus\eps^\me\bDqsom&,\quad s<N/2\\
\bRqsom\dotplus\eps^\me\bDqsom\dotplus\Delta_\eps\eta\towerPqsmz&,\quad s>N/2
\end{cases}\quad.
\end{align*}
Here $\Delta_\eps\eta\towerPqsmz$ is a finite dimensional subspace of
$\qhon{1}{q}{s}(\om)\cap\cq{1}(\om)$\,,
whose elements have supports in the outside of an arbitrarily large ball, and $\eta$ is a cut-off function
as before but now vanishing near the boundary $\dom$\,.
The forms from $\towerPqsmz$ are potential forms, i.e. solve Laplace's equation in $\rN\ohne\{0\}$\,, and
$$\Delta_\eps=\rot\pdiv+\eps^\me\pdiv\rot\qquad.$$
In the special case $\eps=\id$ since $\Delta=\rot\pdiv+\pdiv\rot$
(Here the Laplacian $\Delta$ is to be understood componentwise in Euclidean coordinates.) we have like above
$$\Delta\eta\towerPqsmz=C_{\Delta,\eta}\towerPqsmz=\calS^q_s\subset\cqunom\qquad,$$
where $C_{A,B}:=AB-BA$ denotes the commutator of two operators $A$ and $B$\,.
(For details see Theorem \ref{decotheo}.)
Furthermore, $\Lzqsom$ decomposes for large $s$ into $\Lzqsom\cap\dhqepsom{}^{\bot_\eps}$
and the linear hull of finitely many smooth forms, which have bounded supports.
We note that for all $t\in[-N/2,N/2-1)$
the spaces of Dirichlet forms $\dhqepsom{t}$ coincide.
Moreover, for all $s\in\rz\ohne\tilde{\pI}$ the irrotational forms from $\bRqsom$ resp.
the solenoidal forms from $\bDqsom$ can be represented as rotations resp. divergences, i.e.
$$\bRqsom=\rot\ronqme{s-1}(\om)\qqtext{,}\bDqsom=\pdiv\dqpe{s-1}(\om)$$
hold except of some special values of $s$ or $q$\,. But contrarily to the case $s=0$ for large $s>1+N/2$
we lose integrability properties, if we want to represent forms in $\bRqsom$ resp. $\bDqsom$ by
rotations of solenoidal resp. divergences of irrotational forms. Looking at Theorem \ref{decorotdivtheolarge} we obtain
\begin{align*}
\bRqsom&=\rot\Big(\big(\ronqme{s-1}(\om)\boxplus\eta\towerH{q-1}{s-1}\big)\cap\dqmen{<\Nh}(\om)\Big)\qquad,\\
\bDqsom&=\pdiv\Big(\big(\dqpe{s-1}(\om)\boxplus\eta\towerH{q+1}{s-1}\big)\cap\ronqpenom{<\Nh}\Big)\qquad,
\end{align*}
i.e. the representing solenoidal resp. irrotational forms no longer belong to $\qLzom{q\mp1}{s-1}$ but to $\qLzom{q\mp1}{t}$ for all $t<N/2$\,.
(For details see Theorems \ref{decorotdivtheosmall}, \ref{decorotdivtheolarge} and \ref{decorotdivtheoexcept}.)

If we project onto the orthogonal complement of $\dH{q}{-s}{\eps}(\om)$\,, i.e. of more Dirichlet forms,
we finally obtain even for large $s>N/2$
$$\Lzqsom\cap\dhqepsom{-s}^{\bot_\eps}=\bRqsom\oplus_\eps\eps^\me\bDqsom\qquad.$$

\section{Definitions and preliminaries}

We consider an exterior domain $\Omega\subset\rN$\,, i.e. $\rN\ohne\Omega$ is compact,
as a special smooth Riemannian manifold of dimension $3\leq N\in\nz$\,,
and fix a radius $r_0$ and radii $r_n:=2^nr_0$\,, $n\in\nz$\,,
such that $\rN\ohne\Omega$ is a compact subset of $U_{r_0}:=\setb{x\in\rN}{|x|<r_0}$\,.
Moreover, we choose a cut-off function $\eta$\,, such that {\paulytimeharmetadef} hold.
We then have $\eta=0$ in $U_{r_1}$ and $\eta=1$ in $A_{r_2}:=\setb{x\in\rN}{|x|>r_2}$
and thus $\supp\nabla\eta\subset\ol{A_{r_1}\cap U_{r_2}}$\,.\mylabel{secdef}

Throughout this paper we will use the notations from \cite{paulytimeharm} and \cite{paulystatic}.
Considering alternating differential forms of rank $q\in\zz$ (short $q$-forms) we denote the exterior derivative
$\pd$ by $\rot$ and the co-derivative $\delta=\pm*\pd*$ ($*$\,: Hodge star operator) by $\pdiv$
to remind of the electro-magnetic background. On $\cqunom$ \big(the vector space of all $\cu$-$q$-forms with compact support in $\om$\big)
we have a scalar product \mylabel{defsection}
$$\skp{\Phi}{\Psi}_{\lzqom}:=\intom\Phi\wedge*\bar{\Psi}\qquad\forall\quad\Phi,\Psi\in\cqunom$$
and an induced norm $\norm{\,\cdot\,}_{\lzqom}:=\skp{\,\cdot\,}{\,\cdot\,}_{\lzqom}^{1/2}$\,.
Thus we may define (taking the closure in the latter norm)
$$\Lzqom{}:=\ol{\cqunom}\qquad,$$
the Hilbert space of all square integrable $q$-forms on $\om$\,.
Moreover, due to Stokes' theorem on $\cqunom$ the linear operators $\rot$ and $\pdiv$ are
formally skew adjoint to each other, i.e.
$$\skp{\rot\Phi}{\Psi}_{\lzqpeom}=-\skp{\Phi}{\pdiv\Psi}_{\lzqom}$$
for all $(\Phi,\Psi)\in\cqunom\times\cqpeunom$\,, which gives rise to weak formulations of $\rot$ and $\pdiv$\,.
Using these and the weight function $\rho(r):=(1+r^2)^{1/2}$ for $s\in\rz$
we introduce the following weighted Hilbert spaces (endowed with their natural norms) of $q$-forms

\begin{align*}
\Lzqsom&:=\setb{E\in\Lzqlocom}{\rho^sE\in\Lzqom{}}\qquad,\\
\rqsom&:=\setb{E\in\Lzqsom}{\rot E\in\Lzqpeom{s+1}}\qquad,\\
\dqsom&:=\setb{H\in\Lzqsom}{\pdiv H\in\qLzom{q-1}{s+1}}\qquad.
\end{align*}
All these spaces equal zero if $q\notin\{0,\dots,N\}$\,.
Furthermore, taking the closure in $\rqsom$ we introduce the Hilbert space
$$\ronqsom:=\ol{\cqunom}\qquad,$$
which generalizes the boundary condition of vanishing tangential component
of a $q$-form at the boundary $\dom$\,. More precisely this generalizes the boundary condition $\iota^*E=0$\,, which means that
the pull-back of $E$ on the boundary of $\om$ \big(considered as a $(N-1)$-dimensional Riemannian submanifold of $\omq$\big) vanishes.
Here $\iota:\p\om\hookrightarrow\ol{\om}$ denotes the natural embedding.

A lower left index $0$ indicates vanishing rotation resp. divergence.

For weighted Sobolev spaces $V_s$\,, $s\in\rz$\,, we define
$$V_{<t}:=\bigcap_{s<t}V_s\qquad.$$
We only consider exterior domains $\om$\,, which possess the `Maxwell local compactness property' {\sf MLCP},
i.e. for all $q$ and all $t<s$ the embeddings
$$\ronqsom\cap\dqsom\hookrightarrow\Lzqtom$$
are compact. \big(See {\paulytimeharmmlcpdef} and the literature cited there.\big)

We assume our real valued transformations to be $\tau$-admissible resp.
$\tau$-$\pc{1}$-admissible as defined in {\paulytimeharmdefeps}. This means shortly
that they generate scalar products on $\lzqom$ and are asymptotically the identity mapping.
The parameter $\tau$ always denotes this rate of convergence and the perturbations only have to be $\pc{1}$
in the outside of an arbitrarily large ball. Hence we may choose $r_0$\,,
such that the transformations are $\pc{1}$ in $A_{r_0}$\,.

Let $\eps$ be a $\tau$-$\pc{1}$-admissible transformation on $q$-forms with some $\tau>0$.
We need the finite dimensional vector space of Dirichlet forms
$$\dhqepsom{t}:=\ronqnom{t}\cap\eps^\me\dqnom{t}\qqtext{,}t\in\rz\qquad.$$
(Here we neglect the indices $\eps$ or $t$ in the cases $\eps=\id$ or $t=0$\,.)
Citing {\paulystaticintdirichlet} we have
$$\dhqepsom{-\Nh}=\dhqepsom{}=\dhqepsom{<\Nh-1}$$
and even $\dhqepsom{}=\dhqepsom{<\Nh}$ if $q\notin\{1,N-1\}$\,.
Thus $\dhqepsom{}\subset\Lzqom{-s}$ for $s>1-N/2$ and even for $s>-N/2$ if $q\notin\{1,N-1\}$\,.

Furthermore, for $\rz\ni s>1-N/2$ we introduce the Hilbert spaces
\beq\bDqsom=\dqsnom\cap\dhqepsom{}^\bot\qtext{,}\bRqsom=\ronqsnom\cap\dhqepsom{}^{\bot_\eps}\quad,\mylabel{bDbRdef}\eeq
where we denote by $\bot_\eps$ the orthogonality with respect to the $\skp{\eps\,\cdot\,}{\,\cdot\,}_{\lzqom}$-scalar product,
i.e. the duality between $\Lzqtom$ and $\Lzqom{-t}$\,. If $\eps=\id$ we simply write $\bot:=\bot_{\id}$\,.
The restrictions on the weights $s$ guarantee $\dhqepsom{}\subset\Lzqom{-s}$\,. But if $q\notin\{1,N-1\}$
also $\dhqepsom{}\subset\Lzqom{<\Nh}$ holds and these definitions extend to $s>-N/2$\,.
Since there are no Dirichlet forms $\dhqepsom{}$
for $q\in\{0,N\}$ in these special cases the definitions \eqref{bDbRdef}
may be extended to all $s\in\rz$ and we have
\begin{align*}
\bRom{0}{s}{0}&=\pr{0}{s}{0}{\circ}(\om)=\{0\}
&&,&\bRom{N}{s}{0}&=\pr{N}{s}{0}{\circ}(\om)=\qLzom{N}{s}&&,\\
\bDom{0}{s}{0}&=\pdi{0}{s}{0}{}(\om)=\qLzom{0}{s}
&&,&\bDom{N}{s}{0}&=\pdi{N}{s}{0}{}(\om)=\begin{cases}\{0\}&, s\geq -N/2\\\Lin\{*\Eins\}&, s<-N/2\end{cases}&&.
\end{align*}
Moreover, there are some other characterizations of these spaces.
We remind of the finitely many special smooth forms $\bonqom\subset\ronqnom{}$ and $\bqom\subset\dqnom{}$
presented in {\paulystaticsecgenstatic},
which have compact resp. bounded supports in $\om$ and the properties
\beq\dhqepsom{}\cap\bonqom{}^{\bot_\eps}=\dhqepsom{}\cap\bqom{}^\bot=\{0\}\qquad.\mylabel{dirichletb}\eeq
We note in passing
$$\dim\dhqepsom{}=\dim\dhqom{}=\#\bonqom{}=\#\bqom{}=:d^q\in\nzn\qquad.$$
Using {\paulystaticdomromequi} we see in fact that
\begin{align*}
\bDqsom&=\dqsnom\cap\dhqom{}^\bot=\dqsnom\cap\bonqom^\bot\qquad,\\
\bRqsom&=\ronqsnom\cap\dhqom{}^\bot=\ronqsnom\cap\bqom^\bot
\end{align*}
do not depend on the transformation $\eps$\,.
Since $\bqom$ is only defined for $q\neq1$ the last characterization in the second equation holds only for $q\neq1$\,.
Now the definitions of $\bDqsom$ and $\bRqsom$ extend to arbitrary weights $s\in\rz$ because
the forms $\bonqom,\bqom$ have bounded supports.
We say that $\om$ possesses the `static Maxwell property' {\sf SMP},
if and only if $\om$ has the {\sf MLCP} and the forms $\bonqom,\bqom$ exist.
For instance, the {\sf SMP} is guaranteed for Lipschitz domains $\om$\,.
\big(See {\paulystaticsecgenstatic} and the literature cited there.\big)
We may choose $r_0$\,, such that $\supp b\subset U_{r_0}$ for all $b\in\bonqom\cup\bqom$ and all $q$\,.

Finally for $s>1-N/2$ or $s>-N/2$ and $q\notin\{1,N-1\}$ we put
$$\bLzqsepsom:=\Lzqsom\cap\dhqepsom{}^{\bot_\eps}\qquad.$$

We also need the negative `tower forms' $\turmd{q}{\ell}{\sigma}{m}{-}$\,, $\turmr{q}{\ell}{\sigma}{m}{-}$
for the values $\ell=0,1,2$ and $\sigma\in\nzn$\,, $m\in\{1,\dots,\mu^q_\sigma\}$ from {\paulystaticsectower},
which are harmonic polynomials except of a multiplication by some negative integer power of $r$\,.
These forms are homogeneous of degree $\homg{\ell}{\sigma}{-}:=\ell-\sigma-N$\,,
belong to $\cqu\big(\rNon\big)$ and satisfy the `tower equations'
\begin{align*}
\rot\turmd{q}{0}{\sigma}{m}{-}&=0&&,\qquad&\pdiv\turmr{q+1}{0}{\sigma}{m}{-}&=0&&,\\
\pdiv\turmd{q}{\ell}{\sigma}{m}{-}&=0&&,&\rot\turmr{q+1}{\ell}{\sigma}{m}{-}&=0&&,\\
\rot\turmd{q}{k}{\sigma}{m}{-}&=\turmr{q+1}{k-1}{\sigma}{m}{-}&&,&\pdiv\turmr{q+1}{k}{\sigma}{m}{-}&=\turmd{q}{k-1}{\sigma}{m}{-}&&,
\end{align*}
where $\ell=0,1,2$ and $k=1,2$\,. (We note briefly that we need the positive tower forms of height zero
$\turmd{q}{0}{\sigma}{m}{+}$\,, $\turmr{q}{0}{\sigma}{m}{+}$ in our proofs as well.
But they are not required to formulate our results.)
From {\paulystaticremint} we have for all $\sigma\in\nzn$\,, $m\in\{1,\dots,\mu^q_\sigma\}$ and all $\ell=0,1,2$
as well as all $k\in\nzn$
$$\turmd{q}{\ell}{\sigma}{m}{-}\in\Lzqs(A_1)\quad\equi\quad\turmd{q}{\ell}{\sigma}{m}{-}\in\qh{k}{q}{s}{}(A_1)\quad\equi\quad s<N/2+\sigma-\ell\qquad,$$
which completely determines the integrability properties of our tower forms at infinity.
The same integrability holds true for $\turmr{q}{\ell}{\sigma}{m}{-}$\,.
Moreover, the ground forms (forms of height $0$), which only occur for $1\leq q\leq N-1$\,,
are linear dependent, i.e. we have
\beq\alpha^q_\sigma\cdot\turmr{q}{0}{\sigma}{m}{-}+\ie\alpha^{q'}_\sigma\cdot\turmd{q}{0}{\sigma}{m}{-}=0\qquad,\mylabel{DRlindep}\eeq
where $\alpha^q_\sigma:=(q+\sigma)^{1/2}$ and $q':=N-q$\,.
This motivates to define the harmonic tower forms
\begin{align}
H^q_{\sigma,m}&:=\alpha^q_\sigma\cdot\turmr{q}{0}{\sigma}{m}{-}=-\ie\alpha^{q'}_\sigma\cdot\turmd{q}{0}{\sigma}{m}{-}\mylabel{Hdef}
\intertext{and the potential tower forms}
P^q_{\sigma,m}&:=\alpha^q_\sigma\cdot\turmr{q}{2}{\sigma}{m}{-}+\ie\alpha^{q'}_\sigma\cdot\turmd{q}{2}{\sigma}{m}{-}\qquad.\mylabel{Qdef}
\end{align}
Since $\Delta=\rot\pdiv+\pdiv\rot$ we then obtain
$$\Delta\turmd{q}{\ell}{\sigma}{m}{-}=\Delta\turmr{q}{\ell}{\sigma}{m}{-}=\Delta H^q_{\sigma,m}=\Delta P^q_{\sigma,m}=0\qqtext{,}\ell=0,1\qquad.$$
Here $\Delta$ denotes the componentwise scalar Laplacian in Euclidean coordinates.
We note also that $P^q_{\sigma,m}=H^q_{\sigma,m}=0$ if $q\in\{0,N\}$\,.
Furthermore, for $s\in\rz$ and $\ell=0,1,2$ we introduce the finite dimensional vector spaces
\begin{align*}
\towerD{q,\ell}{s}&:=\Lin\setb{\turmd{q}{\ell}{\sigma}{m}{-}}{\turmd{q}{\ell}{\sigma}{m}{-}\notin\Lzqs(A_1)}=\Lin\set{\turmd{q}{\ell}{\sigma}{m}{-}}{\sigma\leq s-N/2+\ell}\qquad,\\
\towerR{q,\ell}{s}&:=\Lin\setb{\turmr{q}{\ell}{\sigma}{m}{-}}{\turmr{q}{\ell}{\sigma}{m}{-}\notin\Lzqs(A_1)}=\Lin\set{\turmr{q}{\ell}{\sigma}{m}{-}}{\sigma\leq s-N/2+\ell}\qquad,\\
\towerH{q}{s}&:=\Lin\setb{H^q_{\sigma,m}}{H^q_{\sigma,m}\notin\Lzqs(A_1)}=\Lin\set{H^q_{\sigma,m}}{\sigma\leq s-N/2}\qquad,\\
\towerP{q}{s}&:=\Lin\setb{P^q_{\sigma,m}}{P^q_{\sigma,m}\notin\Lzqs(A_1)}=\Lin\set{P^q_{\sigma,m}}{\sigma\leq s-N/2+2}\qquad.
\end{align*}
\big(Here we set $\Lin\emptyset:=\{0\}$\,.\big) We note
$$\towerH{q}{s}=\towerP{q}{s-2}=\towerD{q,\ell}{s-\ell}=\towerR{q,\ell}{s-\ell}=\{0\}\quad\equi\quad s<N/2\qquad.$$
Unfortunately due to the fact that $\rot r^{2-N}$ (the gradient of $r^{2-N}$ in classical terms)
is irrotational and solenoidal but is itself no divergence,
there exist four exceptional tower forms. These are (up to constants)
$$\check{P}^0:=r^{2-N}=\turmd{0}{2}{0}{1}{-}\qqtext{,}
\check{H}^1:=\rot\check{P}^0=r^{1-N}\pd r=\turmr{1}{1}{0}{1}{-}$$
and their dual forms $\check{P}^N:=*\check{P}^0=\turmr{N}{2}{0}{1}{-}$\,, $\check{H}^{N-1}:=*\check{H}^1=\turmd{N-1}{1}{0}{1}{-}$\,.
We then have $\pdiv\check{P}^N=\check{H}^{N-1}$\,.
Following the construction of the regular tower forms we define for $s\geq N/2-2$
\begin{align*}
\towerPh{0}{}&:=\towerPh{0}{s}:=\Lin\{\check{P}^0\}&&,&\towerPh{N}{}&:=\towerPh{N}{s}:=\Lin\{\check{P}^N\}
\intertext{and for $s\geq N/2-1$  }
\towerHh{1}{}&:=\towerHh{1}{s}:=\Lin\{\check{H}^1\}&&,&\towerHh{N-1}{}&:=\towerHh{N-1}{s}:=\Lin\{\check{H}^{N-1}\}\qquad.
\end{align*}
For all other values of $s$ and $q$ we put $\towerPh{q}{s}:=\{0\}$ and $\towerHh{q}{s}:=\{0\}$\,.

As described in {\paulystaticsecstaticoperators} for $s\in\rz$ we will consider vector spaces
$$V^q_s\dotplus\eta\calV^q_s\qqtext{,}\text{($\dotplus$\,: direct sum)}$$
where $V^q_s\subset\Lzqsom$ is some Hilbert space and
$\calV^q_s$ is some finite subset of our tower forms,
e.g. $V^q_s=\ronqsom\cap\dqsom$ and $\calV^q_s=\towerH{q}{s}$\,.
On $V^q_s\dotplus\eta\calV^q_s$ we define a scalar product, such that
\begin{itemize}
\item\quad in $V^q_s$ the original scalar product is kept,
\item\quad $\eta\calV^q_s$ is an orthonormal system,
\item\quad the sum $V^q_s\dotplus\eta\calV^q_s=V^q_s\boxplus\eta\calV^q_s$ is orthogonal.
\end{itemize}
As already indicated we denote the orthogonal sum with respect to this new inner product by $\boxplus$
and clearly $V^q_s\boxplus\eta\calV^q_s$ is a Hilbert space since $\calV^q_s$ is finite.

\section{Results}

Let $\om\subset\rN$ $(N\geq3)$ be an exterior domain as in the last section with the {\sf SMP} or the {\sf MLCP}
depending on whether the forms $\bonqom,\bqom$ are involved in our considerations or not.
Recalling from {\paulystaticsecstaticoperators} the set of special weights $\pI$ we put \mylabel{results}
$$\tilde{\pI}:=\pI-1=\set{N/2+n-1}{n\in\nzn}\cup\set{-N/2-n}{n\in\nzn}$$
and from now on we make the following general assumptions:

\begin{itemize}
\item\quad $q\in\{0,\dots,N\}$
\item\quad $s\in\rz\ohne\tilde{\pI}$\,, i.e. $s+1\in\rz\ohne\pI$\,, i.e. for all $n\in\nzn$
$$s\neq n+N/2-1\qqtext{and}s\neq -n-N/2\qquad.$$
\item\quad $\eps$ is a $\tau$-$\pc{1}$-admissible transformation on $q$-forms with some $\tau=\tau_{s+1}$ satisfying
$$\tau>\max\{0,s+1-N/2\}\qqtext{and}\tau\geq-s-1\qquad,$$
i.e.
$$\tau\,\begin{cases}\,\geq-s-1&,\,s\in(-\infty,-1)\\\,>0&,\,s\in[-1,N/2-1]\\\,>s+1-N/2&,\,s\in(N/2-1,\infty)\end{cases}\qquad.$$
\item\quad $\nu$ and $\mu$ are $\tilde{\tau}$-$\pc{1}$-admissible transformation on $(q-1)$- resp $(q+1)$-forms
with some $\tilde{\tau}=\tau_s$ satisfying
$$\tilde{\tau}>\max\{0,s-N/2\}\qqtext{and}\tilde{\tau}\geq-s\qquad,$$
i.e.
$$\tilde{\tau}\,\begin{cases}\,\geq-s&,\,s\in(-\infty,0)\\\,>0&,\,s\in[0,N/2]\\\,>s-N/2&,\,s\in(N/2,\infty)\end{cases}\qquad.$$
\end{itemize}
If $1-N/2<s<N/2-1$ or $-N/2<s<N/2$ and $q\notin\{1,N-1\}$ the 'trivial' orthogonal decomposition
\beq\Lzqsom=\bLzqsepsom\oplus_\eps\dhqepsom{}\mylabel{trivodeco}\eeq
holds. Throughout the paper we will denote the orthogonality with respect to the $\skp{\eps\,\cdot\,}{\,\cdot\,}_{\lzqom}$-scalar product
or $\Lzqsom$-$\Lzqom{-s}$-duality by $\oplus_\eps$ and put $\oplus=\oplus_{\id}$\,.

The first lemma shows how one may get rid of Dirichlet forms even for larger weights.

\begin{lem}\mylabel{lzdirichletzerl}
Let $s>1-N/2$\,. Then the direct decompositions
$$\Lzqsom=\bLzqsepsom\dotplus\Lin\bonqom\qqtext{,}\Lzqsom=\bLzqsepsom\dotplus\eps^\me\Lin\bqom$$
hold, where the latter is only defined for $q\neq1$\,. If $q\notin\{1,N-1\}$ this decompositions hold for $s>-N/2$\, as well.
\end{lem}

To formulate our main decomposition result we need the operator (a perturbation of the Laplacian $\Delta$)
$$\Delta_\eps:=\rot\pdiv+\eps^\me\pdiv\rot=\Delta+\epsh\pdiv\rot\qquad,$$
where $\eps^\me=:\id+\epsh$ is also $\tau$-$\pc{1}$-admissible. We obtain

\begin{theo}\mylabel{decotheo}
The following decompositions hold:
\begin{itemize}
\item[\bf(i)] If $s<-N/2$\,, then
\begin{align*}
\Lzqsom&=\ronqsnom+\eps^\me\dqsnom
\intertext{and the intersection equals the finite dimensional space of Dirichlet forms
$\dH{q}{s}{\eps}(\om)$\,. Moreover,}
\Lzqsom&=\ronqsnom+\eps^\me\bDqsom
\intertext{and for $q\neq1$ even}
\Lzqsom&=\bRqsom+\eps^\me\bDqsom\qquad.
\end{align*}
In both cases the intersection equals the finite dimensional space of weighted Dirichlet forms
$\dH{q}{s}{\eps}(\om)\cap\bonqom{}^{\bot_\eps}$\,.
\item[\bf(ii)] If $-N/2<s\leq1-N/2$\,, then
\begin{align*}
\qLzsom{1}&=\pr{1}{s}{0}{\circ}(\om)\dotplus\eps^\me\bDom{1}{s}{0}&&,\\
\qLzsom{N-1}&=\bRom{N-1}{s}{0}\dotplus\eps^\me\bDom{N-1}{s}{0}\dotplus\dH{N-1}{}{\eps}(\om)&&.
\end{align*}
\item[\bf(iii)] If $1-N/2<s<N/2$ or $-N/2<s\leq1-N/2$ and $q\notin\{1,N-1\}$\,, then
$$\bLzqsepsom=\bRqsom\dotplus\eps^\me\bDqsom\qquad.$$
For $s\geq0$ this decomposition is even $\skp{\eps\,\cdot\,}{\,\cdot\,}_{\lzqom}$-orthogonal.
\item[\bf(iv)] If $s>N/2$\,, then
\begin{align*}
\bLzqsepsom&=\Big(\big([\Lzqsom\boxplus\eta\towerHqs]\cap\bRom{q}{<\Nh}{0}\big)\\
&\qquad\qquad\oplus_\eps\eps^\me\big([\Lzqsom\boxplus\eta\towerHqs]\cap\bDom{q}{<\Nh}{0}\big)\Big)\cap\Lzqsom
\intertext{and}
\bLzqsepsom&=\bRqsom\dotplus\eps^\me\bDqsom\dotplus\Delta_\eps\eta\towerPqsmz\qquad,
\end{align*}
where the first two terms in the second decomposition are
$\skp{\eps\,\cdot\,}{\,\cdot\,}_{\lzqom}$-orthogonal as well. Furthermore,
$$\Lzqsom\cap\dhqepsom{-s}^{\bot_\eps}=\bRqsom\oplus_\eps\eps^\me\bDqsom\qquad.$$
\end{itemize}
\end{theo}

\begin{rem}\mylabel{decorem}
\begin{itemize}
\item The decompositions in {\bf (ii)}-{\bf (iv)} are direct and define continuous projections.
\item In {\bf (ii)} we are forced to use the forms $\bqom$ and $\bonqom$ in the definitions of $\bRqsom$ and $\bDqsom$\,.
\item To prove the last equation in {\bf (iv)} we additionally assume $\tau\geq N/2-1$\,.
\item The coefficients of the tower forms in the first equation of {\bf (iv)} are related in the following way: If
$$F_{r,s}+\sum_{\ell}h_{r,\ell}\cdot\eta H_\ell+\eps^\me\big(F_{d,s}+\sum_{\ell}h_{d,\ell}\cdot\eta H_\ell\big)=F\in\bLzqsepsom$$
with $F_{r,s}\in\ronqsom$\,, $F_{d,s}\in\dqsom$ and $H_\ell\in\towerHqs$ as well as $h_{r,\ell},h_{d,\ell}\in\cz$\,, then
$$h_{r,\ell}+h_{d,\ell}=0\qquad,$$
since the $H_\ell$ are linear independent and do not belong to $\Lzqsom$\,.
\item $\Delta_\eps\eta\towerPqsmz$ is a finite dimensional subspace of
$\qhon{1}{q}{s}(\om)\cap\cq{1}(\om)$\,,
whose elements have supports in $\ol{A_{r_1}}$\,.
\item For $s<-N/2$ (and $\tau\geq N/2-1$) we have
$$\dH{q}{s}{\eps}(\om)=\dhqepsom{}\dotplus\dH{q}{s}{\eps}(\om)\cap\bonqom{}^{\bot_\eps}\qquad.$$
\item Clearly the transformation $\eps$ may be moved to the $\rot$-free terms in our decompositions as well.
\end{itemize}
\end{rem}

Our decompositions and representations may be refined. For small weights we get

\begin{theo}\mylabel{decorotdivtheosmall}
Let $s<N/2+1-\delta_{q,0}-\delta_{q,N}$\,. Then $\bDqsom$ and $\bRqsom$ are closed subspaces of $\Lzqsom$ whenever they exist and
\begin{align*}
\text{\bf(i)}&&\bRqsom&=\rot\big(\ronqme{s-1}(\om)\cap\nu^\me\bDom{q-1}{s-1}{0}\big)\\
&&&=\rot\big(\ronqme{s-1}(\om)\cap\nu^\me\dqmen{s-1}(\om)\big)=\rot\ronqme{s-1}(\om)
\intertext{\qquad\qquad\qquad holds for $2\leq q\leq N$ as well as for $q=1$ and $s>1-N/2$\,,}
\text{\bf(ii)}&&\bDqsom&=\pdiv\big(\dqpe{s-1}(\om)\cap\mu^\me\bRom{q+1}{s-1}{0}\big)
\intertext{\qquad\qquad\qquad holds for $1\leq q\leq N-1$ as well as for $q=0$ and $s>2-N/2$\,,}
\text{\bf(iii)}&&\bDqsom&=\pdiv\big(\dqpe{s-1}(\om)\cap\mu^\me\ronqpen{s-1}(\om)\big)=\pdiv\dqpe{s-1}(\om)
\intertext{\qquad\qquad\qquad holds for $0\leq q\leq N-1$\,.}
\end{align*}
\end{theo}

For large weights we have

\begin{theo}\mylabel{decorotdivtheolarge}
Let $1\leq q\leq N-1$\,. Then for $s>N/2+1$
\begin{align*}
\text{\bf(i)}&&\bRqsom&=\rot\Big(\big(\ronqme{s-1}(\om)\boxplus\eta\towerH{q-1}{s-1}\big)\cap\nu^\me\bDom{q-1}{<\Nh}{0}\Big)\\
&&&=\rot\big(\ronqme{s-1}(\om)\cap\nu^\me\dqme{s-1}(\om)\cap\pb{q-1}{\circ}(\om)^{\bot_\nu}\big)=\rot\ronqme{s-1}(\om)\quad,\\
\text{\bf(ii)}&&\bDqsom&=\pdiv\Big(\big(\dqpe{s-1}(\om)\boxplus\eta\towerH{q+1}{s-1}\big)\cap\mu^\me\bRom{q+1}{<\Nh}{0}\Big)\\
&&&=\pdiv\big(\dqpe{s-1}(\om)\cap\mu^\me\ronqpe{s-1}(\om)\cap\B^{q+1}(\om)^{\bot_\mu}\big)=\pdiv\dqpe{s-1}(\om)
\intertext{are closed subspaces of $\Lzqsom$ and for $s>N/2$}
\text{\bf(iii)}&&&\qquad\big(\Lzqsom\boxplus\eta\towerHqs\big)\cap\bRom{q}{<\Nh}{0}\\
&&&=\rot\Big(\big(\ronqme{s-1}(\om)\boxplus\eta\towerD{q-1,0}{s-1}\boxplus\eta\towerD{q-1,1}{s-1}\big)\cap\nu^\me\bDom{q-1}{<\Nh-1}{0}\Big)\\
&&&=\bRqsom\dotplus\rot\nu^\me\eta\towerD{q-1,1}{s-1}\qquad,\\
&&&\\
\text{\bf(iv)}&&&\qquad\big(\Lzqsom\boxplus\eta\towerHqs\big)\cap\bDom{q}{<\Nh}{0}\\
&&&=\pdiv\Big(\big(\dqpe{s-1}(\om)\boxplus\eta\towerR{q+1,0}{s-1}\boxplus\eta\towerR{q+1,1}{s-1}\big)\cap\mu^\me\bRom{q+1}{<\Nh-1}{0}\Big)\\
&&&=\bDqsom\dotplus\pdiv\mu^\me\eta\towerR{q+1,1}{s-1}
\end{align*}
are closed subspaces of $\Lzqsom\boxplus\eta\towerHqs$\,.
\end{theo}

\begin{rem}\mylabel{decorotdivremlargetower}
We note $\pdiv\eta\turmd{q-1}{1}{\sigma}{m}{-}=0$ and $\rot\eta\turmr{q+1}{1}{\sigma}{m}{-}=0$ by {\paulystaticoddtower} and thus
$$\eta\towerD{q-1,1}{s-1}\subset\bDom{q-1}{<\Nh-1}{0}\qqtext{,}\eta\towerR{q+1,1}{s-1}\subset\bRom{q+1}{<\Nh-1}{0}\qquad.$$
\end{rem}

\begin{rem}\mylabel{decorotdivremlarge}
Since there are no regular harmonic tower forms in the cases $q\in\{0,N\}$\,,
i.e. $\towerH{0}{s}=\{0\}$\,, $\towerH{N}{s}=\{0\}$\,, and because of $\eta\towerHqs\subset\Lzqom{<\Nh}$
the first equations in {\bf(iii)} and {\bf(iv)} simplify:\\
If $s>N/2$\,, then
\begin{align*}
\big(\qLz{1}{s}(\om)\boxplus\eta\towerH{1}{s}\big)\cap\bRom{1}{<\Nh}{0}
&=\rot\big(\pr{0}{s-1}{}{\circ}(\om)\boxplus\eta\towerD{0,1}{s-1}\big)\qquad,\\
\big(\qLz{N-1}{s}(\om)\boxplus\eta\towerH{N-1}{s}\big)\cap\bDom{N-1}{<\Nh}{0}
&=\pdiv\big(\pdi{N}{s-1}{}{}(\om)\boxplus\eta\towerR{N,1}{s-1}\big)\qquad.
\end{align*}
If $N/2<s<N/2+1$\,, then
\begin{align*}
\big(\Lzqsom\boxplus\eta\towerHqs\big)\cap\bRom{q}{<\Nh}{0}
&=\rot\Big(\big(\ronqme{s-1}(\om)\boxplus\eta\towerD{q-1,1}{s-1}\big)\cap\nu^\me\bDom{q-1}{<\Nh-1}{0}\Big)\quad,\\
\big(\Lzqsom\boxplus\eta\towerHqs\big)\cap\bDom{q}{<\Nh}{0}
&=\pdiv\Big(\big(\dqpe{s-1}(\om)\boxplus\eta\towerR{q+1,1}{s-1}\big)\cap\mu^\me\bRom{q+1}{<\Nh-1}{0}\Big)\quad.
\end{align*}
\end{rem}

As already mentioned above there are no harmonic tower forms in the remaining cases $q\in\{0,N\}$\,.
Thus the equations in (i), (ii) and (iii), (iv) of the latter theorem would coincide for these values.
Furthermore, in these special cases there occur the exceptional tower forms. We obtain

\begin{theo}\mylabel{decorotdivtheoexcept}
Let $s>N/2$\,. Then
\begin{align*}
\text{\bf(i)}&&\qLz{N}{s}(\om)&=\bRom{N}{s}{0}\\
&&&=\rot\Big(\big(\pr{N-1}{s-1}{}{\circ}(\om)\boxplus\eta\towerH{N-1}{s-1}\boxplus\eta\towerHh{N-1}{}\big)\cap\nu^\me\bDom{N-1}{<\Nh-1}{0}\Big)\\
&&&=\rot\Big(\big(\pr{N-1}{s-1}{}{\circ}(\om)\cap\nu^\me\pdi{N-1}{s-1}{}{}(\om)\cap\pb{N-1}{\circ}(\om)^{\bot_\nu}\big)\boxplus\eta\towerHh{N-1}{}\Big)\\
&&&=\rot\big(\pr{N-1}{s-1}{}{\circ}(\om)\cap\nu^\me\pdi{N-1}{s-1}{}{}(\om)\cap\pb{N-1}{\circ}(\om)^{\bot_\nu}\big)\dotplus\Delta\eta\towerPh{N}{}\\
&&&=\rot\pr{N-1}{s-1}{}{\circ}(\om)\dotplus\Delta\eta\towerPh{N}{}\qquad,\\
\\
\text{\bf(ii)}&&\qLz{0}{s}(\om)&=\bDom{0}{s}{0}\\
&&&=\pdiv\Big(\big(\pdi{1}{s-1}{}{}(\om)\boxplus\eta\towerH{1}{s-1}\boxplus\eta\towerHh{1}{}\big)\cap\mu^\me\bRom{1}{<\Nh-1}{0}\Big)\\
&&&=\pdiv\Big(\big(\pdi{1}{s-1}{}{}(\om)\cap\mu^\me\pr{1}{s-1}{}{\circ}(\om)\big)\boxplus\eta\towerHh{1}{}\Big)\\
&&&=\pdiv\big(\pdi{1}{s-1}{}{}(\om)\cap\mu^\me\pr{1}{s-1}{}{\circ}(\om)\big)\dotplus\Delta\eta\towerPh{0}{}\\
&&&=\pdiv\pdi{1}{s-1}{}{}(\om)\dotplus\Delta\eta\towerPh{0}{}\qquad.
\end{align*}
\end{theo}

Finally we note

\begin{rem}\mylabel{endrem}
We always get easily dual results using the Hodge star operator.
This would change the homogeneous boundary condition from the tangential (electric) to the
normal (magnetic) one. However, since this would multiply the number of results by two
we let their formulation to the interested reader.
\end{rem}

\section{Proofs}

Let $\eps$\,, $\nu$ and $\mu$ be as in section \ref{results}. We start with the

\vspace*{2mm}
\begin{proof}{\bf of Lemma \ref{lzdirichletzerl} }
Let $E\in\Lzqsom$\,. Looking at {\paulystaticsecgenstatic} and using the Helmholtz decompositions {\paulytimeharmlzzerl}
we may choose (smooth) $q$-forms $b_\ell\in\Lin\bonqom$\,, $\ell=1,\dots,d^q$\,,
with $b_\ell=\Phi_\ell+H_\ell\in\ol{\rot\pR{q-1}{}{}{\circ}(\om)}\oplus_\eps\dhqepsom{}$\,, where $\{H_\ell\}$ is a $\oplus_\eps$-ONB
of $\dhqepsom{}$\,. Then we have $E-\sum_\ell\skp{E}{\eps H_\ell}_{\lzqom}b_\ell\in\bLzqsepsom$\,. This proves
one inclusion and the other one is trivial, because the forms of $\bonqom$ are smooth and compactly supported.
Moreover, if $E\in\Lin\bonqom\cap\dhqepsom{}^{\bot_\eps}$\,, then $\eps E=\sum_\ell e_\ell\eps b_\ell\in\dhqepsom{}^{\bot}$ and thus
$$0=\skp{\eps E}{H_k}_{\lzqom}=\sum_\ell e_\ell\skp{\eps H_\ell}{H_k}_{\lzqom}=e_k\qquad,$$
which proves the directness of the sum. The other direct decomposition may be shown in a similar way.
\end{proof}

We introduce the Hilbert spaces \big(closed subspaces of $\Lzqsom$\big)
\begin{align*}
\ronqssom&:=\setb{E\in\ronqsom}{\rot(\rho^{2s}E)=0}\\
&\;=\setb{E\in\ronqsom}{\rot E=-2s\rho^{-2}RE}\qquad,\\
\dqssom&:=\setb{E\in\dqsom}{\pdiv(\rho^{2s}E)=0}\\
&\;=\setb{E\in\dqsom}{\pdiv E=-2s\rho^{-2}TE}
\end{align*}
and note the important fact
that the $\norm{\,\cdot\,}_{\rqsom}$-, $\norm{\,\cdot\,}_{\Rqsom}$- and $\norm{\,\cdot\,}_{\Lzqsom}$-norms resp.
the $\norm{\,\cdot\,}_{\dqsom}$-, $\norm{\,\cdot\,}_{\Dqsom}$- and $\norm{\,\cdot\,}_{\Lzqsom}$-norms are equivalent on
$\ronqssom$ resp. $\dqssom$\,.
Here we used the operators $R,T$ from \cite[Definition 1]{sphharm}.
First we need an easy consequence of the projection theorem:

\begin{lem}\mylabel{lzsorthodeco}
Let $s\in\rz$\,. Then the orthogonal decompositions
\begin{align*}
\text{\bf(i)}&&\Lzqsom&=\ol{\rot\pr{q-1}{s-1}{}{\circ}(\om)}\oplus_{s,\eps}\eps^\me\dqssom=\eps^\me\ol{\rot\pr{q-1}{s-1}{}{\circ}(\om)}\oplus_{s,\eps}\dqssom\qquad,\\
\text{\bf(ii)}&&\Lzqsom&=\ol{\pdiv\pdi{q+1}{s-1}{}{}(\om)}\oplus_{s,\eps}\eps^\me\ronqssom=\eps^\me\ol{\pdiv\pdi{q+1}{s-1}{}{}(\om)}\oplus_{s,\eps}\ronqssom
\end{align*}
hold with continuous projections. Here we denote by $\oplus_{s,\eps}$ the orthogonal sum
with respect to the $\skp{\eps\rho^{2s}\,\cdot\,}{\,\cdot\,}_{\lzqom}$-scalar product and the closures
are taken in $\Lzqsom$\,.
The space $\ol{\rot\pr{q-1}{s-1}{}{\circ}(\om)}$ resp.
$\ol{\pdiv\pdi{q+1}{s-1}{}{}(\om)}$ may be replaced by $\ol{\rot\qcom{\infty}{q-1}{\circ}}$ resp. $\ol{\pdiv\pdi{q+1}{{\vox}}{}{}(\om)}$\,.
\end{lem}

\begin{proof}
Since $\qcom{\infty}{q-1}{\circ}$ is dense in $\pr{q-1}{s-1}{}{\circ}(\om)$ we have
$E\in\Lzqsom\cap\big(\rot\pr{q-1}{s-1}{}{\circ}(\om)\big)^{\bot_{s,\eps}}$\,, if and only if
$E\in\Lzqsom\cap\big(\rot\qcom{\infty}{q-1}{\circ}\big)^{\bot_{s,\eps}}$\,, which means $\rho^{2s}\eps E\in\dqnom{-s}$\,.
Thus $E\in\dqssom$\,, because $0=\pdiv(\rho^{2s}\eps E)=\rho^{2s}\pdiv\eps E+2s\rho^{2s-2}T\eps E$\,.
This shows (i) and (ii) follows analogously.
\end{proof}

\begin{rem}\mylabel{lzsorthodecorem}
Clearly this lemma holds for $0$-admissible transformations $\eps$ as well.
\end{rem}

We need two important results, which may be formulated as follows:
Defining the Hilbert space
$$\sideset{_\eps}{^q_t}\X(\om):=\big(\ronqtom\cap\eps^\me\dqtom\big)\boxplus\eta\towerH{q}{t}\boxplus\eta\towerHh{q}{t}\qqtext{,}t\in\rz\qquad,$$
we have

\begin{lem}\mylabel{rotdivmap}
Let $s\in\rz\ohne\tilde{\pI}$ and $q\neq0$ or $q=0$ and $s>-N/2$\,. Then
$$\Abb{\sideset{_\eps}{^q_{s}}\ROT}{\sideset{_\eps}{^q_{s}}\X(\om)\cap\,\eps^\me\dqlocnom}{\bRom{q+1}{s+1}{0}}{E}{\rot E}\qquad,$$
$$\Abb{\sideset{_\eps}{^q_{s}}\DIV}{\sideset{_\eps}{^q_{s}}\X(\om)\cap\,\ronqlocnom}{\bDom{q-1}{s+1}{0}}{H}{\pdiv\eps H}$$
are continuous and surjective Fredholm operators with kernels
$$N(\sideset{_\eps}{^q_{s}}\ROT)=N(\sideset{_\eps}{^q_{s}}\DIV)=\begin{cases}\dhqepsom{s}&\text{ , if }s<-N/2\\\dhqepsom{}&\text{ , if }s>-N/2\end{cases}\qquad.$$
\end{lem}

\begin{rem}\mylabel{rotdivmaprem}
By adding suitable Dirichlet forms from $\dhqepsom{}$ we always may obtain $H\bot\bqom$\,, $q\neq1$\,, and $\eps E\bot\bonqom$ or
$\eps H,\eps E\bot\dhqepsom{}$\,, if $s>-N/2$\,.
Then for $s>-N/2$ the operators
$$\sideset{_\eps}{^q_{s}}\ROT:\sideset{_\eps}{^q_{s}}\X(\om)\cap\eps^\me\bDom{q}{\loc}{0}\To\bRom{q+1}{s+1}{0}\qquad,$$
$$\sideset{_\eps}{^q_{s}}\DIV:\sideset{_\eps}{^q_{s}}\X(\om)\cap\bRom{q}{\loc}{0}\To\bDom{q-1}{s+1}{0}$$
are also injective and hence topological isomorphisms by the bounded inverse theorem.
Furthermore, we note \big(using the notations from \cite{paulystatic}\big)
$$\towerH{q}{t}=\calD^q(\pcIbq{0}{t})=\calR^q(\pcJbq{0}{t})\qqtext{,}\towerHh{q}{t}=\check{\calD}^{q,1}_{t}=\check{\calR}^{q,1}_{t}$$
and for $t<N/2$
$$\towerH{q}{t}=\towerHh{q}{t-1}=\{0\}\qquad.$$
\end{rem}

\begin{proof}
This lemma has been proved in {\paulystaticrotdivcorlem} for $s>-N/2$\,.
Hence we only have to discuss the cases of small weights
$s<-N/2$ and ranks of forms $1\leq q\leq N$\,.
Welldefinedness, continuity and the assertions about the kernels are trivial in these cases.
To show surjectivity for $\sideset{_\eps}{^q_{s}}\DIV$ let us pick some $F\in\bDom{q-1}{s+1}{0}$\,.
Following the proofs of {\paulystaticproofdiv} and using \cite[Theorem 4]{sphharm}
we represent the extension by zero of $F$ to $\rN$ by
$$\hat{F}=:F_\textrm{D}+F_\textrm{R}\in\pdi{q-1}{s+1}{0}{}(\rN)+\pr{q-1}{s+1}{0}{}(\rN)\qquad.$$
Now $F_\textrm{D}$ is contained in the range of the operator $B$ from \cite[Theorem 7]{sphharm} and thus we get some
$$h\in\pdq{s}(\rN)\cap\rqn{s}(\rN)\qqtext{solving}\pdiv h=F_\textrm{D}\qquad.$$
Applying the regularity result {\paulydissregtheoaussenzwei} we even have $h\in\qh{1}{q}{s}{}(\rN)$\,.
Since $F_\textrm{R}$ is an element of $\pdi{q-1}{s+1}{0}{}(\om)\cap\pr{q-1}{s+1}{0}{}(\om)$
we may represent this form in terms of a spherical harmonics expansion
$$\restr{F_\textrm{R}}{A_{r_0}}=\sum_{I\in\pcIb{q-1}{0}{}}f_ID^{q-1}_I+f\hat{D}^{q-1,1}
+\sum_{I\in\pcIpqmens}f_ID^{q-1}_I$$
with uniquely determined $f_I,f\in\cz$ using {\paulystaticsphdeco}, where
$$\pcIpqmens:=\setb{I\in\pcI{q-1}{0}{}}{\s(I)=+\,\wedge\,\ei(I)<-s-1-N/2}\qquad.$$
Now looking at {\paulystaticremint} the first term of the sum on the right hand side belongs to
$\qLz{q-1}{<\Nh}(A_{r_0})$\,, the second to $\qLz{q-1}{<\Nh-1}(A_{r_0})$
and the third to $\qLz{q-1}{s+1}(A_{r_0})$\,. Therefore,
$$F_\textrm{R}-\sum_{I\in\pcIpqmens}f_I\eta D^{q-1}_I\in\qLz{q-1}{}(\rN)\qquad.$$
This suggests the ansatz
$$H:=\eta h+\sum_{I\in\pcIpqmens}f_I\eta R^{q}_{{}_1I}+\Phi$$
to solve $H\in\ronqnom{s}\cap\eps^\me\dqom{s}$ and $\pdiv\eps H=F$\,.
Thus we are searching for some $q$-form $\Phi$ in $\ronqom{s}\cap\eps^\me\dqom{s}$ satisfying
\begin{align}\begin{split}
\rot\Phi&=-\rot(\eta h)=-C_{\rot,\eta}h=:\tilde{G}\qquad,\\
\pdiv\eps\Phi&=F-\pdiv(\eta\eps h)-\sum_{I\in\pcIpqmens}f_I\pdiv(\eta\eps R^{q}_{{}_1I})=:\tilde{F}\qquad,
\end{split}\mylabel{transsys}\end{align}
since $\rot(\eta R^{q}_{{}_1I})=0$\,. Clearly we have $\tilde{G}\in\bRom{q+1}{\vox}{0}\subset\bRom{q+1}{}{0}$\,.
Moreover, not only $\tilde{F}$ is an element of $\bDom{q-1}{s+1}{0}$\,, but also $\tilde{F}\in\bDom{q-1}{}{0}$ holds, because
\begin{align*}
\tilde{F}&=\pdiv(1-\eta)h-\sum_{I\in\pcIpqmens}f_IC_{\pdiv,\eta}R^{q}_{{}_1I}
+F_\textrm{R}-\sum_{I\in\pcIpqmens}f_I\eta D^{q-1}_I\\
&\qquad-\pdiv\big(\epsd\eta(h+\sum_{I\in\pcIpqmens}f_I\eta R^{q}_{{}_1I})\big)\qquad,
\end{align*}
where the first two terms of the sum on the right hand side lie in $\qLz{q-1}{\vox}(\om)$\,,
the sum of the third and fourth terms in $\qLz{q-1}{}(\om)$
and the last one in $\qLz{q-1}{s+1+\tau}(\om)\subset\qLz{q-1}{}(\om)$\,, since
$$\eta(h+\sum_{I\in\pcIpqmens}f_I\eta R^{q}_{{}_1I})\in\qh{1}{q}{s}{}(\rN)$$
and $\tau\geq-s-1$\,. Now we are able to apply {\paulydisstopiso}.
Doing this we get some $\Phi\in\ronqom{-1}\cap\eps^\me\dqom{-1}$
solving the system \eqref{transsys}. Since $s<-N/2<-3/2$ we have $\Phi\in\ronqom{s}\cap\eps^\me\dqom{s}$\,,
which completes the proof. The assertion about $\sideset{_\eps}{^q_{s}}\ROT$ follows analogously.
\end{proof}

\begin{proof}{\bf of Theorem \ref{decorotdivtheosmall} } 
Apply Lemma \ref{rotdivmap} and Remark \ref{rotdivmaprem} with the modified values of $q$\,, $s$ and $\eps$\,.
\end{proof}

\begin{proof}{\bf of Theorem \ref{decorotdivtheolarge} }
The first equations in (i)-(iv) have already been proved in {\paulystaticLloes}. Let
$$G\in\bRqsom=\rot\Big(\big(\ronqme{s-1}(\om)\boxplus\eta\towerH{q-1}{s-1}\big)\cap\nu^\me\bDom{q-1}{<\Nh}{0}\Big)\qquad,$$
i.e.
$$G=\rot G_{s-1}+\sum_{I\in\pcIb{q-1}{0}{s-1}}g_I\rot\eta D^{q-1}_I$$
with $G_{s-1}\in\ronqme{s-1}(\om)\cap\nu^\me\dqme{s-1}(\om)\cap\pb{q-1}{\circ}(\om)^{\bot_\nu}$ and $g_I\in\cz$\,. Since
\begin{align*}
\rot\eta D^{q-1}_I&=\rot\Delta\eta D^{q-1}_{{}_2I}-\rot C_{\pdiv\rot,\eta}D^{q-1}_{{}_2I}
\intertext{and clearly}
\rot\Delta\eta R^{q-1}_{{}_2I}&=\rot\pdiv C_{\rot,\eta}R^{q-1}_{{}_2I}
\intertext{we obtain}
\ie\alpha^{q'}_{\ei(I)}\rot\eta D^{q-1}_I&=\rot\Delta\eta P^{q-1}_{\ei(I),\ci(I)}-\ie\alpha^{q'}_{\ei(I)}\rot C_{\pdiv\rot,\eta}D^{q-1}_{{}_2I}
-\alpha^q_{\ei(I)}\rot\pdiv C_{\rot,\eta}R^{q-1}_{{}_2I}\qquad.
\end{align*}
Now $\Delta\eta P^{q-1}_{\ei(I),\ci(I)}=C_{\Delta,\eta}P^{q-1}_{\ei(I),\ci(I)}$ has compact support and therefore
$$G\in\rot\big(\ronqme{s-1}(\om)\cap\nu^\me\dqme{s-1}(\om)\cap\pb{q-1}{\circ}(\om)^{\bot_\nu}\big)\qquad,$$
which proves (i). (ii) is shown analogously.
The last equation in (iii) follows from (i),
since we can split off a term
$$\eta\turmd{q-1}{1}{\sigma}{m}{-}=\nu^\me\nu\eta\turmd{q-1}{1}{\sigma}{m}{-}=\nu^\me\eta\turmd{q-1}{1}{\sigma}{m}{-}+\nu^\me\hat{\nu}\eta\turmd{q-1}{1}{\sigma}{m}{-}$$
and the tower forms are smooth and $\hat{\nu}$ decays
as well as $\pdiv\eta\turmd{q-1}{1}{\sigma}{m}{-}=0$ holds by Remark \ref{decorotdivremlargetower}.
Finally the last equation in (iv) is a direct consequence of (ii)
and a similar argument like the latter one.
\end{proof}

\begin{proof}{\bf of Theorem \ref{decorotdivtheoexcept} }
Lemma \ref{rotdivmap} yields
$$\qLzom{N}{s}=\rot\Big(\big(\pr{N-1}{s-1}{}{\circ}(\om)\boxplus\eta\towerH{N-1}{s-1}\boxplus\eta\towerHh{N-1}{}\big)\cap\nu^\me\bDom{N-1}{<\Nh-1}{0}\Big)$$
and the same arguments used in the latter proof show
$$\rot\Big(\big(\pr{N-1}{s-1}{}{\circ}(\om)\boxplus\eta\towerH{N-1}{s-1}\big)\cap\nu^\me\bDom{N-1}{<\Nh-1}{0}\Big)
=\rot\big(\pr{N-1}{s-1}{}{\circ}(\om)\cap\nu^\me\pdi{N-1}{s-1}{}{}(\om)\cap\pb{N-1}{\circ}(\om)^{\bot_\nu}\big)\,.$$
Because $\pdiv\eta\check{H}^{N-1}=0$ we have
$$\qLzom{N}{s}=\rot\Big(\big(\pr{N-1}{s-1}{}{\circ}(\om)\cap\nu^\me\pdi{N-1}{s-1}{}{}(\om)\cap\pb{N-1}{\circ}(\om)^\bot\big)\boxplus\eta\towerHh{N-1}{}\Big)\qquad.$$
Now $\pdiv\check{P}^N=\check{H}^{N-1}$ and thus
$\rot\eta\check{H}^{N-1}=\Delta\eta\check{P}^N-\rot C_{\pdiv,\eta}\check{P}^N$\,,
which proves
$$\qLzom{N}{s}=\rot\big(\pr{N-1}{s-1}{}{\circ}(\om)\cap\nu^\me\pdi{N-1}{s-1}{}{}(\om)\cap\pb{N-1}{\circ}(\om)^{\bot_\nu}\big)\dotplus\Delta\eta\towerPh{N}{}\qquad.$$
Finally we have to show that the sum is direct. To do this, let $F=f\Delta\eta\check{P}^N=\rot E$
with some $f\in\cz$ and $E\in\pr{N-1}{s-1}{}{\circ}(\om)$\,.
We want to use the notations from \cite{sphharm}, i.e. the forms $P^{N,4}_{0,1}$ and $Q^{N,4}_{0,1}$\,, as well.
By definition there exists a constant $c\neq0$\,,
such that $\check{P}^N=c\turmr{N}{2}{0}{1}{-}=\frac{c}{2-N}Q^{N,4}_{0,1}$
using {\paulystaticremsph}.
Since $s>N/2$ and $P^{N,4}_{0,1}\in\qLzom{N}{<-\Nh}$ partial integration yields
$$\skp{\rot E}{P^{N,4}_{0,1}}_{\lzqom}=0\qquad,$$
because $P^{N,4}_{0,1}\in\Lin\{*\Eins\}$ is constant.
But on the other hand we obtain
$$\skp{F}{P^{N,4}_{0,1}}_{\lzqom}=f\skp{\Delta\eta\check{P}^N}{P^{N,4}_{0,1}}_{\lzqom}
=\frac{c\,f}{2-N}\skp{C_{\Delta,\eta}Q^{N,4}_{0,1}}{P^{N,4}_{0,1}}_{\lzqom}=c\,f$$
by \cite[(73)]{sphharm}, i.e. $f=0$\,.
\end{proof}

Now we turn to the main idea of our decompositions and the

\vspace*{2mm}
\begin{proof}{\bf of Theorem \ref{decotheo} }
Let $1\leq q\leq N-1$ and $s$ be as in section \ref{results} as well as $F\in\Lzqsom$\,.
Using Lemma \ref{lzsorthodeco} we decompose $F=F_r+\eps^\me\hat{F}_d$
with $F_r\in\ol{\rot\qcom{\infty}{q-1}{\circ}}$ and $\hat{F}_d\in\dqssom$\,.
A second application of this lemma yields the decomposition $\eps^\me\hat{F}_d=\eps^\me F_d+\tilde{F}$
with $F_d\in\ol{\pdiv\dqpevoxom}$ and $\tilde{F}\in\ronqsom\cap\eps^\me\dqsom$\,.
Furthermore, there exists a constant $c>0$ independent of $F$\,, such that
$$\norm{F_r}_{\Lzqsom}+\norm{F_d}_{\Lzqsom}+\norm{\tilde{F}}_{\rqsom\cap\eps^\me\dqsom}
\leq c\,\norm{F}_{\Lzqsom}\qquad.$$
Now $\tilde{F}$ is more regular than $F$ and this enables us to solve
$$\pdiv\eps H=\pdiv\eps\tilde{F}\in\bDom{q-1}{s+1}{0}\qqtext{,}\rot E=\rot\tilde{F}\in\bRom{q+1}{s+1}{0}$$
with some $H\in\sideset{_\eps}{^q_{s}}\X(\om)\cap\,\ronqlocnom$
and $E\in\sideset{_\eps}{^q_{s}}\X(\om)\cap\,\eps^\me\dqlocnom$ by Lemma \ref{rotdivmap}.
We note $E,H\in\Lzqtom$ for all $t$ with
$t\leq s$ and $t<N/2-1$\,. Then $\hat{F}:=\tilde{F}-E-H\in\dhqepsom{t}$ and
\beq F=F_r+H+\eps^\me(F_d+\eps E)+\hat{F}\qquad,\mylabel{Frepone}\eeq
where $F_r+H\in\ronqtnom$ and $F_d+\eps E\in\dqtnom$\,.

For $s>-N/2$ we may refine this representation of $F$\,.
In fact for these values of $s$ we have $\hat{F}\in\dhqepsom{>-\Nh}=\dhqepsom{}$
by {\paulystaticintdirichlet}. Using Remark \ref{rotdivmaprem} or {\paulystaticrotdivtheospez} additionally we may obtain
$\eps E\bot\bonqom$ and $H\bot\bqom$\,, if $q\neq1$\,, or $\eps E\bot\dhqepsom{}$ and $\eps H\bot\dhqepsom{}$\,, if $s>1-N/2$\,.
Therefore, $F_d+\eps E\in\bDom{q}{t}{0}$ holds for $s>-N/2$ and $F_r+H\in\bRom{q}{t}{0}$ for $s>1-N/2$ or $1-N/2>s>-N/2$ and $q\neq1$\,.
Moreover, {\paulystaticrotdivtheospez} yields not only $E,H\in\sideset{_\eps}{^q_{s}}\X(\om)$ but also
\begin{align*}
E&\in\big(\ronqsom\cap\eps^\me\dqsom\big)\boxplus\eta\calD^q(\pcIbq{0}{s})&&\text{, if }q\neq N-1,\\
H&\in\big(\ronqsom\cap\eps^\me\dqsom\big)\boxplus\eta\calR^q(\pcJbq{0}{s})&&\text{, if }q\neq 1,
\end{align*}
i.e. the exceptional forms do not appear in these cases.

The stronger assumption $F\in\bLzqsepsom$ for $s>1-N/2$ or $s>-N/2$ and $2\leq q\leq N-2$ implies $\hat{F}\in\dhqepsom{}^{\bot_\eps}$ and
thus $\hat{F}=0$\,. Hence in these cases \eqref{Frepone} turns to
\beq F=F_r+H+\eps^\me(F_d+\eps E)\qquad.\mylabel{Freptwo}\eeq
Until now we have shown the assertions of Theorem \ref{decotheo} (i), (ii) and also (iii) for weights $s<N/2-1$\,.

Considering larger weights $s>N/2-1$ and $F\in\bLzqsepsom$ the tower forms occur in the representation \eqref{Freptwo}.
More precisely we have
\begin{align*}
E&=E_s+\sum_{I\in\pcIbq{0}{s}}e_I\eta D^q_I+e\eta\begin{cases}\turmd{N-1}{1}{0}{1}{-}&\text{, }q=N-1\\0&\text{, otherwise}\end{cases}\qquad,\\
H&=H_s+\sum_{J\in\pcJbq{0}{s}}h_J\eta R^q_J+h\eta\begin{cases}\turmr{1}{1}{0}{1}{-}&\text{, }q=1\\0&\text{, otherwise}\end{cases}
\end{align*}
with uniquely determined forms $E_s,H_s\in\ronqsom\cap\eps^\me\dqsom$ and complex coefficients $e_I,e,h_J,h$\,.
We note $\pcIbq{0}{s}=\pcJbq{0}{s}$\,.
For $s<N/2$ we have $\pcIbq{0}{s}=\emptyset$ and
for $s>N/2$ we see $\alpha^q_{\ei(I)}\eta R^q_I=-\ie\alpha^{q'}_{\ei(I)}\eta D^q_I\in\Lzqom{<\Nh}$
for all $I\in\pcIbq{0}{s}$ by \eqref{DRlindep}.
Thus we obtain
\begin{align*}
F&=F_r+H_s+\eps^\me F_d+E_s\\
&\qquad\qquad+\sum_{I\in\pcIbq{0}{s}}(h_I-\tilde{e}_I)\eta R^q_I
+\eta\begin{cases}h\turmr{1}{1}{0}{1}{-}&\text{, }q=1\\e\turmd{N-1}{1}{0}{1}{-}&\text{, }q=N-1\\0&\text{, otherwise}\end{cases}\qquad,
\end{align*}
where $\tilde{e}_I:=-\ie e_I\alpha^q_{\ei(I)}/\alpha^{q'}_{\ei(I)}$\,.
Looking, for example, at $q=1$ we find $\eta\turmr{1}{1}{0}{1}{-}\notin\Lzqom{\geq\Nh-1}$\,.
Then for integrability reasons we get $h=0$\,, such that the exceptional tower form does not appear.
Clearly also $e=0$ holds true for $q=N-1$\,.
Moreover, $h_I=\tilde{e}_I$ since $R^q_I$ are linear independent and $\eta R^q_I\notin\Lzqom{s}$ for all $I\in\pcIbq{0}{s}$ and $s>N/2$\,.

By the smoothness of $\eta D^q_I$ as well as the decay and differentiability properties of $\epsd$ we obtain furthermore
$\epsd\eta D^q_I\in\qhom{1}{q}{s}{}$\,. Thus for all $s>1-N/2$ we get the representation
$$F=\tilde{F}_r+\sum_{I\in\pcIbq{0}{s}}h_I\eta R^q_I+\eps^\me\big(\tilde{F}_d+\sum_{I\in\pcIbq{0}{s}}e_I\eta D^q_I\big)\qquad,$$
where $\tilde{F}_r:=F_r+H_s$ and $\bds\tilde{F}_d:=F_d+\eps E_s+\sum_{I\in\pcIbq{0}{s}}e_I\epsd\eta D^q_I\eds$ as well as
\begin{align*}
\tilde{F}_r+\sum_{I\in\pcIbq{0}{s}}h_I\eta R^q_I&\in\big(\Lzqsom\boxplus\eta\towerHqs\big)\cap\bRom{q}{<\Nh}{0}\qquad,\\
\tilde{F}_d+\sum_{I\in\pcIbq{0}{s}}e_I\eta D^q_I&\in\big(\Lzqsom\boxplus\eta\towerHqs\big)\cap\bDom{q}{<\Nh}{0}
\end{align*}
with $\alpha^{q'}_{\ei(I)}h_I+\ie\alpha^q_{\ei(I)}e_I=0$ for all $I\in\pcIbq{0}{s}$\,.
This proves the remaining assertions of (iii) and the first equation in (iv).

To show the second equation in (iv) we observe
\begin{align*}
\eta D^q_I&=\eta\pdiv\rot D^q_{{}_2I}=\pdiv\rot\eta D^q_{{}_2I}-C_{\pdiv\rot,\eta}D^q_{{}_2I}\qquad,\\
\eta R^q_I&=\eta\rot\pdiv R^q_{{}_2I}=\rot\pdiv\eta R^q_{{}_2I}-C_{\rot\pdiv,\eta}R^q_{{}_2I}
\end{align*}
and therefore
$$F=\tilde{\tilde{F}}_r+\eps^\me\tilde{\tilde{F}}_d+\sum_{I\in\pcIbq{0}{s}}\frac{h_I}{\alpha^q_{\ei(I)}}\Delta_\eps\eta P^q_{\ei(I),\ci(I)}\qquad,$$
where
\begin{align*}
\tilde{\tilde{F}}_r:=\tilde{F}_r-\sum_{I\in\pcIbq{0}{s}}h_IC_{\rot\pdiv,\eta}R^q_{{}_2I}-\sum_{I\in\pcIbq{0}{s}}e_I\rot C_{\pdiv,\eta}D^q_{{}_2I}&\in\bRqsom\qquad,\\
\tilde{\tilde{F}}_d:=\tilde{F}_d-\sum_{I\in\pcIbq{0}{s}}e_IC_{\pdiv\rot,\eta}D^q_{{}_2I}-\sum_{I\in\pcIbq{0}{s}}h_I\pdiv C_{\rot,\eta}R^q_{{}_2I}&\in\bDqsom\qquad,\\
\sum_{I\in\pcIbq{0}{s}}\frac{h_I}{\alpha^q_{\ei(I)}}\Delta_\eps\eta P^q_{\ei(I),\ci(I)}&\in\Delta_\eps\eta\towerPqsmz\qquad.
\end{align*}
Clearly all sums are direct resp. orthogonal as stated. Only in the second equation of (iv) one may see this not directly.
So, for example, if
$$E=\sum_{I\in\pcIbq{0}{s}}e_I\Delta_\eps\eta P^q_{\ei(I),\ci(I)}=G+\eps^\me F\in\bRqsom\dotplus\eps^\me\bDqsom$$
with some $s>N/2$\,, then
$$H:=G-\sum_{I\in\pcIbq{0}{s}}e_I\rot\pdiv\eta P^q_{\ei(I),\ci(I)}=\sum_{I\in\pcIbq{0}{s}}e_I\eps^\me\pdiv\rot\eta P^q_{\ei(I),\ci(I)}-\eps^\me F$$
is a Dirichlet form, i.e. $H\in\dhqepsom{}$\,, but also an element of $\bonqom^{\bot_\eps}$\,. Hence $H$ must vanish and thus
$$G=\sum_{I\in\pcIbq{0}{s}}e_I\rot\pdiv\eta P^q_{\ei(I),\ci(I)}\in\Lzqsom\qquad,$$
which is only possible, if $e_I=0$ for all $I\in\pcIbq{0}{s}$\,,
since $\rot\pdiv\eta P^q_{\ei(I),\ci(I)}\notin\Lzqsom$ are linear independent.

It remains to prove the last equation of (iv). Before we start with this we observe
that by the closed graph theorem all projections in (ii)-(iv) are continuous. We note
$$\bRqsom,\eps^\me\bDqsom\subset\Lzqsom\cap\dhqepsom{-s}^{\bot_\eps}=:\bYqsom$$
and thus
$$\bXqsom:=\bRqsom\oplus_\eps\eps^\me\bDqsom\subset\bYqsom\qquad.$$
Furthermore, $\bXqsom$ and $\bYqsom$ are closed subspaces of $\Lzqsom$\,.
By the first equation of (iv) and Lemma \ref{lzdirichletzerl} we have
$$\codim\bXqsom=\dim\dhqepsom{}+\dim\Delta_\eps\eta\towerPqsmz=d^q+\sum_{0\leq\sigma<s-\Nh}\mu^q_\sigma\qquad,$$
since $\dim\Delta_\eps\eta\towerPqsmz=\dim\towerPqsmz=\dim\towerHqs$
and $s-N/2\notin\nzn$ because $s\notin\tilde{\pI}$\,.
With the identity $\codim\bYqsom=\dim\dhqepsom{-s}$ we get by Appendix A that $\bXqsom$ and $\bYqsom$
possess the same finite codimension in $\Lzqsom$\,. Consequently we obtain $\bXqsom=\bYqsom$\,.
\end{proof}

\appendix

\section{Appendix}

\subsection{Weighted Dirichlet forms}

Let $\tau>0$\,. As already mentioned in section \ref{secdef} for the space of Dirichlet forms we have
$$\dhqepsom{-\Nh}=\dhqepsom{}=\dhqepsom{<t}\qqtext{,}t:=N/2-\delta_{q,1}-\delta_{q,N-1}\qquad,$$
and its dimension equals $d^q=\beta_{q'}$\,, the $q'$th Betti number of $\Omega$\,. Furthermore, we have for all $t\in\rz$
$$\dH{0}{t}{\eps}(\om)=\{0\}\qqtext{,}
\dH{N}{t}{\eps}(\om)=\begin{cases}\{0\}&, t\geq-N/2\\\eps^\me\Lin\{*\Eins\}&, t<-N/2\end{cases}\qquad.$$
(This holds even for $\tau=0$\,.)
We repeat some notations and results from \cite{paulydiss}, \cite{paulystatic} and \cite{sphharm}.
Let us introduce the `special growing Dirichlet forms' $\Esm$ from {\paulydissEsmdef} or \cite{paulyasym} as the unique solutions of the problems
$$\Esm\in\dhqepsom{<-\Nh-\sigma}\cap\bonqom^{\bot_\eps}\qqtext{,}\Esm-\turmd{q}{0}{\sigma}{m}{+}\in\Lzqom{>-\Nh}\qquad,$$
where $\sigma\in\nzn$ and $1\leq m\leq\mu^q_\sigma$ with
$$\mu^q_\sigma=\binom{N}{q}\binom{N-1+\sigma}{\sigma}\frac{qq'(N+2\sigma)}{N(q+\sigma)(q'+\sigma)}$$
from \cite[Theorem 1 (iii)]{sphharm}. To guarantee their existence we have to impose the decay conditions
$\tau>\sigma$ and $\tau\geq N/2-1$\,. We note $\mu^q_0=\binom{N}{q}$ and thus $\mu^0_0=\mu^N_0=1$\,.
Moreover, $\turmd{0}{0}{0}{1}{+}$ resp. $\turmd{N}{0}{0}{1}{+}$ is a multiple of $\Eins$ resp. $*\Eins$\,.

\begin{lem}\mylabel{weighteddirichletforms}
Let $1\leq q\leq N-1$ and $s\in(-\infty,-N/2)\ohne\tilde{\pI}$ as well as $\tau>-s-N/2$ and $\tau\geq N/2-1$\,. Then
$$\dhqepsom{s}=\dhqepsom{}\dotplus\dhqepsom{s}\cap\bonqom^{\bot_\eps}$$
holds. Moreover,
$$\dhqepsom{s}\cap\bonqom^{\bot_\eps}=\Lin\set{E^q_{\sigma,m}}{\sigma<-s-N/2}\qquad.$$
\end{lem}

\begin{cor}\mylabel{weighteddirichletformscor}
The dimension $d^q_s$ of $\dhqepsom{s}$ is finite and independent of $\eps$\,. More precisely
$$d^q_s=d^q+\sum_{\sigma<-s-\Nh}\mu^q_\sigma\qquad.$$
Furthermore, the locally constant mapping
$$\Abb{d^q(\,\cdot\,)}{(-\infty,-N/2)\ohne\tilde{\pI}\cup(-N/2,N/2-1)}{\nzn}{s}{d^q_s}$$
is monotone decreasing. It jumps exactly at the points $s\in\tilde{\pI}$\,, i.e. $-s-N/2\in\nzn$\,.
\end{cor}

\begin{proof}
The directness of the sum follows by \eqref{dirichletb} and the inclusions
\begin{align*}
\dhqepsom{}\dotplus\dhqepsom{s}\cap\bonqom^{\bot_\eps}&\subset\dhqepsom{s}\qquad,\\
\Lin\set{E^q_{\sigma,m}}{\sigma<-s-N/2}&\subset\dhqepsom{s}\cap\bonqom^{\bot_\eps}
\intertext{are trivial. So it remains to prove}
\dhqepsom{s}&\subset\dhqepsom{}\dotplus\Lin\set{E^q_{\sigma,m}}{\sigma<-s-N/2}\qquad.
\end{align*}
Therefore, we pick some $E\in\dhqepsom{s}$\,. We observe
$E\in\qh{1}{q}{s}{}(A_\rho)$ by the regularity result {\paulydissregtheoaussenlokal} and even
$$\rot E=0\qqtext{,}\restr{\pdiv E}{A_\rho}=-\restr{\pdiv \epsd E}{A_\rho}\in\Lzq{s+\tau+1}(A_\rho)$$
for all $r_0<\rho<r_1$\,. Thus we have $\eta E\in\ronqsom\cap\dqsom$ with
\begin{align*}
\rot\eta E=C_{\rot,\eta}E&\in\ronqpevoxnom\cap\bqpeom^\bot\qquad,\\
\pdiv\eta E=C_{\pdiv,\eta}E-\eta\pdiv E&\in\pdi{q-1}{s+\tau+1}{0}{}(\om)\cap\bonqmeom^\bot\qquad.
\end{align*}
The assumptions on $\tau$ yield $s+\tau+1>1-N/2$\,. Thus by Lemma \ref{rotdivmap} there exists $e\in\ronqtom\cap\dqtom$
with some $t>-N/2$ solving
$$\rot e=\rot\eta E\qqtext{,}\pdiv e=\pdiv\eta E\qquad.$$
Therefore, $H:=\eta E-e\in\dH{q}{s}{}(\om)$\,.
Thus $\rot H=0$ and $\pdiv H=0$ in $A_{r_0}$ and we may represent $H$ in terms of a spherical harmonics expansion
$$\restr{H}{A_{r_0}}=\sum_{\gamma,n}h^-_{\gamma,n}\turmd{q}{0}{\gamma}{n}{-}+\hat{h}\hat{D}^{q,1}
+\sum_{\sigma<-s-\Nh}\sum_{m=1}^{\mu^q_\sigma}h^+_{\sigma,m}\turmd{q}{0}{\sigma}{m}{+}$$
with uniquely determined $h^-_{\gamma,n},\hat{h},h^+_{\sigma,m}\in\cz$ using {\paulystaticsphdeco}.
By {\paulystaticremint} the first term of the sum on the right hand side belongs to
$\Lzq{<\Nh}(A_{r_0})$\,, the second to $\Lzq{<\Nh-1}(A_{r_0})$
and the third to $\Lzqs(A_{r_0})$\,. We get
\begin{align*}
H-\sum_{\sigma<-s-\Nh}\sum_{m=1}^{\mu^q_\sigma}e_{\sigma,m}\turmd{q}{0}{\sigma}{m}{+}&\in\Lzq{<\Nh-1}(A_{r_0})\qquad,
\intertext{which yields}
h:=H-\sum_{\sigma<-s-\Nh}\sum_{m=1}^{\mu^q_\sigma}h^+_{\sigma,m}\Esm&\in\Lzqom{>-\Nh}\qquad.
\end{align*}
Finally we obtain
$$E-\sum_{\sigma<-s-\Nh}\sum_{m=1}^{\mu^q_\sigma}h^+_{\sigma,m}\Esm=(1-\eta)E+e+h\in\dhqepsom{>-\Nh}=\dhqepsom{}\qquad.$$
\end{proof}

\subsection{Vector fields in three dimensions}

Now we will translate our results to the classical framework of vector analysis.
Thus we switch to some (maybe) more common notations.

Let $N:=3$\,. We identify $1$-forms with vector fields via Riesz' representation theorem and
$2$-forms with $1$-forms via the Hodge star operator
and thus with vector fields as well.
Using Euclidean coordinates $\{x_1,x_2,x_3\}$ this means in detail
we identify the vector field
$$E=\dvec{E_1}{E_2}{E_3}$$
with the $1$-form
$$E_1\pd x_1+E_2\pd x_2+E_3\pd x_3$$
resp. with the $2$-form
$$E_1*\pd x_1+E_2*\pd x_2+E_3*\pd x_3
=E_1\pd x_2\wedge\pd x_3+E_2\pd x_3\wedge\pd x_1+E_3\pd x_1\wedge\pd x_2\quad.$$
Moreover, we identify the $3$-form $E\pd x_1\wedge\pd x_2\wedge\pd x_3$
with the $0$-form and/or function $E$\,.
We will denote these identification isomorphisms by $\cong$\,.
Then the exterior derivative and co-derivative turn to the classical differential operators
$$\grad=\nabla=\dvec{\p_1}{\p_2}{\p_3}\qqtext{,}\curl=\nabla\times\qqtext{,}\pdiv=\nabla\cdot$$
from vector analysis, where $\times$ resp. $\cdot$ denotes the vector resp. scalar product in $\rd$\,.
In particular we have the following identification table:
\begin{table}[h!]
\begin{center}
\fbox{\begin{tabular}{|c||c|c|c|c|}
\hline
& $q=0$ & $q=1$ & $q=2$ & $q=3$\\
\hline\hline
$\rot=\pd$ & $\grad$ & $\curl$ & $\pdiv$ & $0$\\
\hline
$\pdiv=\delta$ & $0$ & $\pdiv$ & $-\curl$ & $\grad$\\
\hline
\end{tabular}}
\end{center}
\caption{\label{identificationtable}Identification table}
\end{table}

\subsubsection{Tower functions and fields}

Let us briefly construct our tower forms once more in classical terms.
Using polar coordinates $\{r,\varphi,\vartheta\}$\,, i.e.
$$x=\Phi(r,\varphi,\vartheta)=r\dvec{\cos\varphi\cos\vartheta}{\sin\varphi\cos\vartheta}{\sin\vartheta}\qquad,$$
we have with an obvious notation
$$\dvec{\pd x_1}{\pd x_2}{\pd x_3}=J_\Phi\dvec{\pd r}{\pd\varphi}{\pd\vartheta}
=Q\dvec{\pd r}{r\cos\vartheta\pd\varphi}{r\pd\vartheta}\qquad,$$
where $Q:=[\ei_r\,\ei_\varphi\,\ei_\vartheta]$ is an orthonormal matrix and
$$\ei_r:=\dvec{\cos\vartheta\cos\varphi}{\cos\vartheta\sin\varphi}{\sin\vartheta}\qtext{,}\ei_\varphi:=\dvec{-\sin\varphi}{\cos\varphi}{0}
\qtext{,}\ei_\vartheta:=\dvec{-\sin\vartheta\cos\varphi}{-\sin\vartheta\sin\varphi}{\cos\vartheta}$$
the corresponding orthonormal basis of $\rd$\,. Since $\{\pd x_1,\pd x_2,\pd x_3\}$
is an orthonormal basis of $1$-forms, $\{\pd r,r\cos\vartheta\pd\varphi,r\pd\vartheta\}$ is an orthonormal basis
as well. Moreover, we have again with an obvious notation
$$\dvec{\pd r}{r\cos\vartheta\pd\varphi}{r\pd\vartheta}=Q^t\dvec{\pd x_1}{\pd x_2}{\pd x_3}
=\dvec{\ei_r^t}{\ei_\varphi^t}{\ei_\vartheta^t}\dvec{\pd x_1}{\pd x_2}{\pd x_3}
=\dvec{\ei_r\cdot\pd x}{\ei_\varphi\cdot\pd x}{\ei_\vartheta\cdot\pd x}\qquad,$$
which shows
$$\pd r\cong\ei_r\qqtext{,}r\cos\vartheta\pd\varphi\cong\ei_\varphi
\qqtext{,}r\pd\vartheta\cong\ei_\vartheta\qquad.$$
We will denote the representations of
$\grad$\,, $\curl$ and $\pdiv$ in polar coordinates by $\gradf$\,, $\curlf$ and $\pdivf$ as well as
their realizations on $S:=S^2$ by $\gradfs$\,, $\curlfs$ and $\pdivfs$\,.
These may be derived by the formula
$$(\nabla_x u)\circ\Phi=J_\Phi^{-t}\nabla_{r,\varphi,\vartheta}(u\circ\Phi)
=[\ei_r\quad(r\cos\vartheta)^\me\ei_\varphi\quad r^\me\ei_\vartheta]
\nabla_{r,\varphi,\vartheta}(u\circ\Phi)$$
and we then have the following representations:
\begin{align*}
\gradf u&=\ei_r\p_r u+\frac{1}{r}\gradfs u&&,&\gradfs u&=\frac{1}{\cos\vartheta}\ei_\varphi\p_\varphi u+\ei_\vartheta\p_\vartheta u\\
\curlf v&=\ei_r\times\p_r v+\frac{1}{r}\curlfs v&&,&\curlfs v&=\frac{1}{\cos\vartheta}\ei_\varphi\times\p_\varphi v+\ei_\vartheta\times\p_\vartheta v\\
\pdivf v&=\ei_r\cdot\p_r v+\frac{1}{r}\pdivfs v&&,&\pdivfs v&=\frac{1}{\cos\vartheta}\ei_\varphi\cdot\p_\varphi v+\ei_\vartheta\cdot\p_\vartheta v
\end{align*}
We note that we do not distinguish between $u$ resp. $v$ and $u\circ\Phi$ resp. $v\circ\Phi$ anymore.
Moreover, in polar coordinates the Laplacian reads
$$\Deltaf=\p_r^2+\frac{2}{r}\p_r+\frac{1}{r^2}\Deltafs\qquad,$$
where
$$\Deltafs=\frac{1}{\cos^2\vartheta}\p_\varphi^2+\p_\vartheta^2-\frac{\sin\vartheta}{\cos\vartheta}\p_\vartheta$$
is the Laplace-Beltrami operator.

Let us introduce the classical spherical harmonics of order $n$
$$y_{n,m}\qqtext{,}n\in\nzn\,,\,m=1,\dots,2n+1\qquad,$$
which form a complete orthonormal system in $\lz(S)$\,, i.e. $\skp{y_{n,m}}{y_{\ell,k}}_{\lz(S)}=\delta_{n,\ell}\delta_{m,k}$\,,
and satisfy
$$\big(\Deltafs+\lambda_n\big)y_{n,m}=0\qqtext{,}\lambda_n:=n(n+1)\qquad,$$
as well as the corresponding potential functions
$$z_{\pm,n,m}:=r^{\theta_{\pm,n}}y_{n,m}\qqtext{,}n\in\nzn\,,\,m=1,\dots,2n+1\qquad,$$
which are homogeneous of degree
$$\theta_{\pm,n}:=\begin{cases}n&\text{, if }\pm=+\\-n-1&\text{, if }\pm=-\end{cases}$$
and solve
$$\Delta z_{\pm,n,m} =\Deltaf z_{\pm,n,m} =0\qquad.$$
\big(See for example \cite[Kapitel VII, � 4]{couranthilberteins} or \cite[chapter 2.3]{coltonkress}.\big)

Moreover, for $k,n\in\nzn$\,, $m=1,\dots,2n+1$ we define
$$z_{\pm,n,m}^k:=\xi_{\pm,n}^{k}r^{2k}z_{\pm,n,m}=\xi_{\pm,n}^{k}r^{\theta_{\pm,n}^{2k}}y_{n,m}\qquad,$$
where $\bds\xi_{\pm,n}^k:=\frac{\Gamma(1\pm n\pm1/2)}{4^k\cdot k!\cdot\Gamma(k+1\pm n\pm1/2)}\eds$ and $\Gamma$ denotes the gamma-function.
The $z_{\pm,n,m}^k$ are homogeneous of degree $\theta_{\pm,n}^{2k}$ with
$$\theta_{\pm,n}^\ell:=\ell+\theta_{\pm,n}$$
and satisfy
$$\Delta z_{\pm,n,m}^k=z_{\pm,n,m}^{k-1}\qquad,$$
where $z_{\pm,n,m}^{-1}:=0$\,. With the aid of these functions, which we will call a '$\Delta$-tower',
we construct for $k\in\nzn$ the functions and fields
\begin{align*}
U^{2k}_{\pm,n,m}&:=z_{\pm,n,m}^k\qquad,\\
U^{2k-1}_{\pm,n,m}&:=\grad U^{2k}_{\pm,n,m}\qquad,
\intertext{which we will call a '$\pdiv\grad$-tower', as well as the fields}
V^{2k}_{\pm,n,m}&:=r\ei_r\times\grad z_{\pm,n,m}^k=\ei_r\times\gradfs z_{\pm,n,m}^k
=\xi_{\pm,n}^{k}r^{\theta_{\pm,n}^{2k}}\ei_r\times Y_{n,m}\qquad,\\
V^{2k-1}_{\pm,n,m}&:=-\curl V^{2k}_{\pm,n,m}\qquad,
\end{align*}
which we will call a '$-\curl\curl$-tower'. Here $Y_{n,m}:=\gradfs y_{n,m}$\,.
The fields $U^{2k-1}_{\pm,n,m}$ are irrotational and the fields $V^\ell_{\pm,n,m}$ solenoidal.
Moreover, we have
$$\pdiv U^{2k+1}_{\pm,n,m}=U^{2k}_{\pm,n,m}\qqtext{,}\curl V^{2k+1}_{\pm,n,m}=V^{2k}_{\pm,n,m}$$
as well as $\pdiv U^{-1}_{\pm,n,m}=0$ and $\curl V^{-1}_{\pm,n,m}=0$ and thus
\begin{align*}
\Delta U^{2k}_{\pm,n,m}&=\pdiv\grad U^{2k}_{\pm,n,m}=U^{{2k}-2}_{\pm,n,m}\qquad,\\
\Delta U^{2k+1}_{\pm,n,m}&=\grad\pdiv U^{2k+1}_{\pm,n,m}=U^{2k-1}_{\pm,n,m}\qquad,\\
\Delta V^\ell_{\pm,n,m}&=-\curl\curl V^\ell_{\pm,n,m}=V^{\ell-2}_{\pm,n,m}\qquad,
\end{align*}
where $U^{-2}_{\pm,n,m}:=0$\,, $V^{-2}_{\pm,n,m}:=0$\,.
We mention that $U^\ell_{\pm,n,m}$ and $V^\ell_{\pm,n,m}$ are homogeneous of degree $\theta_{\pm,n}^\ell$\,.
Moreover,
$$U^{-1}_{\pm,n,m}=V^{-1}_{\pm,n,m}\qquad.$$
Thus we define
$$P_{\pm,n,m}:=U^{1}_{\pm,n,m}-V^{1}_{\pm,n,m}\qquad.$$
Then $U^\ell_{\pm,n,m}$\,, $V^\ell_{\pm,n,m}$\,, $\ell=-1,0$ and even $P_{\pm,n,m}$ are potential fields resp. functions.

The next picture may illustrate the denotations tower:
\begin{figure}[h!]
$$\begin{array}{|r||ccrclcc|}
\dots&&&\dots\,\qquad&\Big|&\,\dots&&\\
&&&\text{\rm\scriptsize div}\downarrow\,\qquad&\Big|&\,\downarrow\text{\rm\scriptsize curl}&&\\
\text{\rm 3. floor}&&&U^{2}_{\pm,n,m}&\Big|&V^{2}_{\pm,n,m}&\xrightarrow{\pdiv}&0\\
&&&\text{\rm\scriptsize grad}\downarrow\,\qquad&\Big|&\,\downarrow\text{\rm\scriptsize $-$curl}&&\\
\text{\rm 2. floor}&0&\xleftarrow{\curl}&U^{1}_{\pm,n,m}&\Big|&V^{1}_{\pm,n,m}&\xrightarrow{\pdiv}&0\\
&&&\text{\rm\scriptsize div}\downarrow\,\qquad&\Big|&\,\downarrow\text{\rm\scriptsize curl}&&\\
\text{\rm 1. floor}&&&U^{0}_{\pm,n,m}&\Big|&V^{0}_{\pm,n,m}&\xrightarrow{\pdiv}&0\\
&&&\text{\rm\scriptsize grad}\downarrow\,\qquad&\Big|&\,\downarrow\text{\rm\scriptsize $-$curl}&&\\
\text{\rm ground}&0&\xleftarrow{\curl}&U^{-1}_{\pm,n,m}&=&V^{-1}_{\pm,n,m}&\xrightarrow{\pdiv}&0\\
&&&\text{\rm\scriptsize div}\downarrow\,\qquad&\Big|&\,\downarrow\text{\rm\scriptsize curl,div}&&\\
&&&0\,\qquad&\Big|&\,0&&\\
\hline\hline
&\multicolumn{3}{c}{\text{\sf $\pdiv\grad$-tower}}&\Big|&\multicolumn{3}{c|}{\text{\sf $-\curl\curl$-tower}}\\
\hline
\end{array}$$
\caption{\label{towerspic}Towers}
\end{figure}\\
In the exceptional case $(n,m)=(0,1)$ the function $z^k_{\pm,0,1}$ is a multiple of $r^{2k+\theta_{\pm,0}}$ and thus
we get $V^\ell_{\pm,0,1}=0$ for all $\ell$ as well as an exceptional $\pdiv\grad$-tower
\begin{align*}
U^{2k}_{\pm,0,1}&=\xi_{\pm,0}^{k}r^{2k-\delta_{-,\pm}}\qquad,\\
U^{2k-1}_{\pm,0,1}&=\grad U^{2k}_{\pm,0,1}=\xi_{\pm,0}^{k}\p_rr^{2k-\delta_{-,\pm}}\ei_r=\xi_{\pm,0}^{k}(2k-\delta_{-,\pm})r^{2k-1-\delta_{-,\pm}}\ei_r\qquad,
\end{align*}
where $U^{-1}_{+,0,1}=0$\,.

Let us briefly compare these classical towers with the $q$-form towers:
On $S$ we identify $1$-forms with linear combinations of $\ei_\varphi$ and $\ei_\vartheta$
as well as $2$-forms with scalar functions.
More precisely a $1$-form $\omega_\varphi\cos\vartheta\pd\varphi+\omega_\vartheta\pd\vartheta$
will be identified with the tangential vector field $\omega_\varphi\ei_\varphi+\omega_\vartheta\ei_\vartheta$
and a $2$-form $\omega\cos\vartheta\pd\varphi\wedge\pd\vartheta$ with the function $\omega$\,.
Then our operators $\rho,\tau$ and $\rhoh,\tauh$ from \cite{sphharm} turn to
\begin{table}[h!]
\begin{center}
\fbox{\begin{tabular}{|c||c|c|c|c|}
\hline
& $q=0$ & $q=1$ & $q=2$ & $q=3$\\
\hline\hline
$\rho v\cong$ & $0$ & $v\cdot\ei_r$ & $v\times\ei_r$ & $v$\\
\hline
$\tau v\cong$ & $v$ & $-(v\times\ei_r)\times\ei_r$ & $v\cdot\ei_r$ & $0$\\
\hline
$\rhoh v\cong$ & $v\ei_r$ & $-v\times\ei_r$ & $v$ & ---\\
\hline
$\tauh v\cong$ & $v$ & $v$ & $v\ei_r$ & ---\\
\hline
\end{tabular}}\qquad,
\end{center}
\caption{\label{sphop}Spherical operators}
\end{table}\\
where
\begin{align*}
-(v\times\ei_r)\times\ei_r&=v\cdot\ei_\varphi\ei_\varphi+v\cdot\ei_\vartheta\ei_\vartheta
=v-v\cdot\ei_r\ei_r\qquad,\\
v\times\ei_r&=v\cdot\ei_\vartheta\ei_\varphi-v\cdot\ei_\varphi\ei_\vartheta\qquad.
\end{align*}
We note
\begin{align*}
y_{0,1}&\cong S^0_{0,1}\qquad\text{is constant},\\
y_{n,m}&\cong T^0_{n-1,m}\qquad,\\
Y_{n,m}=\gradfs y_{n,m}&\cong\ie n^{1/2}(n+1)^{1/2}S^1_{n-1,m}
\end{align*}
for $n\in\nz$ in the terminology of \cite{sphharm}.
Then for $n\in\nz$ we get up to constants
\begin{align*}
U^{2k}_{\pm,n,m}&\cong D^0_{(\pm,2k+1,n-1,m)}=*R^3_{(\pm,2k+1,n-1,m)}\qquad,\\
U^{2k-1}_{\pm,n,m}&\cong R^1_{(\pm,2k,n-1,m)}=*D^2_{(\pm,2k,n-1,m)}\qquad,\\
V^\ell_{\pm,n,m}&\cong D^1_{(\pm,\ell+1,n-1,m)}=*R^2_{(\pm,\ell+1,n-1,m)}
\intertext{and for $n=0$}
U^{2k}_{+,0,1}&\cong D^0_{(+,2k,0,1)}=*R^3_{(+,2k,0,1)}\qquad,\\
U^{2k-1}_{+,0,1}&\cong R^1_{(+,2k-1,0,1)}=*D^2_{(+,2k-1,0,1)}\qquad,\\
U^{2k}_{-,0,1}&\cong D^0_{(-,2k+2,0,1)}=*R^3_{(-,2k+2,0,1)}\qquad,\\
U^{2k-1}_{-,0,1}&\cong R^1_{(-,2k+1,0,1)}=*D^2_{(-,2k+1,0,1)}\qquad.
\end{align*}
Finally for $s\in\rz$ and $\ell=-1,0$ we put \big(with $\Lin\emptyset:=\{0\}$\big)
\begin{align*}
\towerV{\ell}{s}&:=\Lin\setb{V^\ell_{-,n,m}}{V^\ell_{-,n,m}\notin\Lzs(A_1)}=\Lin\set{V^\ell_{-,n,m}}{n\leq\ell+s+1/2}\quad,\\
\towerU{\ell}{s}&:=\Lin\setb{U^\ell_{-,n,m}}{U^\ell_{-,n,m}\notin\Lzs(A_1)}=\Lin\set{U^\ell_{-,n,m}}{n\leq\ell+s+1/2}\quad,\\
\towerP{}{s}&:=\Lin\setb{P_{-,n,m}}{P_{-,n,m}\notin\Lzs(A_1)}=\Lin\set{P_{-,n,m}}{n\leq s+3/2}\quad,\\
\towerUh{\ell}{}:=\towerUh{\ell}{s}&:=\Lin\setb{U^\ell_{-,0,1}}{U^\ell_{-,0,1}\notin\Lzs(A_1)}=\Lin\set{U^\ell_{-,0,1}}{0\leq\ell+s+1/2}\quad.
\end{align*}

\subsubsection{Results for vector fields}

For some operator $\Diamond\in\{\grad,\curl,\pdiv\}$ and $s\in\rz$ we define the Hilbert spaces \mylabel{clHdef}
\begin{align*}
\clH{s}(\Diamond,\om)&:=\setb{u\in\Lzsom}{\Diamond u\in\Lzom{s+1}}&&,&
\clH{s}(\overset{\circ}{\Diamond},\om)&:=\ol{\cunom}\quad,\\
\clH{s}(\Diamond_0,\om)&:=\setb{\clH{s}(\Diamond,\om)}{\Diamond u=0}&&,&
\clH{s}(\overset{\circ}{\Diamond}_0,\om)&:=\setb{\clH{s}(\overset{\circ}{\Diamond},\om)}{\Diamond u=0}\quad,
\end{align*}
where the closure is taken in $\clH{s}(\Diamond,\om)$\,.
Then the spaces $\ronqsom$ and $\dqsom$ turn to the usual Sobolev spaces, i.e.:
\begin{table}[h!]
\begin{center}
\fbox{\begin{tabular}{|c||c|c|c|c|}
\hline
& $q=0$ & $q=1$ & $q=2$ & $q=3$\\
\hline\hline
$\ronqsom$ & $\clH{s}(\gradon,\om)=\clHon{s}^1(\om)$ & $\clH{s}(\curlon,\om)$ & $\clH{s}(\pdivon,\om)$ & $\Lzsom$\\
\hline
$\dqsom$ & $\Lzsom$ & $\clH{s}(\pdiv,\om)$ & $\clH{s}(\curl,\om)$ & $\clH{s}(\grad,\om)=\clH{s}^1(\om)$\\
\hline
\end{tabular}}
\end{center}
\caption{\label{sobolevsp}Sobolev spaces}
\end{table}\\
For two operators $\Diamond,\Box\in\big\{\overset{(\circ)}{\grad}_{(0)},\overset{(\circ)}{\curl}_{(0)},\overset{(\circ)}{\pdiv}_{(0)}\big\}$ we define
$$\clH{s}(\Diamond,\Box,\om):=\clH{s}(\Diamond,\om)\cap\clH{s}(\Box,\om)\qquad.$$
The generalized boundary condition $\iota^*E=0$ for a $q$-form $E$ from $\ronqlocom$ turns to the usual boundary
conditions $\gamma E=\restr{E}{\p\om}=0$\,, $\gamma_t E=\restr{\nu\times E}{\p\om}=0$ and $\gamma_n E=\restr{\nu\cdot E}{\p\om}=0$
(for $q=0,1,2$) weakly formulated in the spaces $\clH{}(\gradon,\om)$\,, $\clH{}(\curlon,\om)$ and $\clH{}(\pdivon,\om)$\,,
where $\nu$ denotes the outward unit normal at $\p\om$ and $\gamma$ the trace as well as $\gamma_t$ resp. $\gamma_n$
the tangential resp. normal trace of the vector field $E$\,.
The linear transformations $\eps$\,, $\nu$\,, $\mu$ ($\nu$ and $\mu$ may be identified!)
can be considered as real valued, variable, symmetric and uniformly positive definite matrices
with $\text{L}^\infty(\om)$-entries, which satisfy the asymptotics at infinity assumed in section 2 and 3.
Moreover, for $\Diamond,\Box\in\big\{\overset{(\circ)}{\curl}_{(0)},\overset{(\circ)}{\pdiv}_{(0)}\big\}$ we define
$$\clH{s}(\Box\eps,\om):=\eps^\me\clH{s}(\Box,\om)\qtext{,}\clH{s}(\Diamond,\Box\eps,\om):=\clH{s}(\Diamond,\om)\cap\clH{s}(\Box\eps,\om)\quad.$$
Now we have two kinds of Dirichlet fields. The first ones, the classical Dirichlet fields,
$$\dH{}{s}{\eps}(\om):=\clH{s}(\curlonn,\pdivn\eps,\om)\cong\dH{1}{s}{\eps}(\om)\qqtext{,}s\in\rz$$
correspond to $q=1$ and the second ones, the classical Neumann fields,
$$\nH{}{s}{\eps}(\om):=\clH{s}(\curln,\pdivonn\eps,\om)\cong\eps^\me\dH{2}{s}{\eps^\me}(\om)\qqtext{,}s\in\rz$$
correspond to $q=2$\,.
Moreover, we have the compactly supported fields
$$\pb{1}{\circ}(\om)\cong:\calBon{1}(\om)\subset\clH{}(\curlonn,\om)\qqtext{,}\pb{2}{\circ}(\om)\cong:\calBon{2}(\om)\subset\clH{}(\pdivonn,\om)$$
and the fields with bounded supports
$$\B^2(\om)\cong:\calB(\om)\subset\clH{}(\curln,\om)\qquad.$$

Let $s>-1/2$\,. Using the Dirichlet and Neumann fields we put
$$\clbH{s}(\Diamond,\om):=\clH{s}(\Diamond,\om)\cap\dH{}{}{}{}(\om)^\bot\qqtext{,}
\clbHt{s}(\Diamond,\om):=\clH{s}(\Diamond,\om)\cap\nH{}{}{}{}(\om)^\bot$$
and define in the same way $\clbH{s}(\Diamond,\Box,\om)$\,, $\clbH{s}(\Diamond,\Box\eps,\om)$ and
$\clbHt{s}(\Diamond,\Box,\om)$\,, $\clbHt{s}(\Diamond,\Box\eps,\om)$\,.
Then we have
\begin{align*}
\clbH{s}(\pdivn,\om)&=\clH{s}(\pdivn,\om)\cap\dH{}{}{\eps}{}(\om)^\bot=\clH{s}(\pdivn,\om)\cap\calBon{1}{}(\om)^\bot&&,\\
\clbHt{s}(\pdivonn,\om)&=\clH{s}(\pdivonn,\om)\cap\nH{}{}{\eps}{}(\om)^\bot=\clH{s}(\pdivonn,\om)\cap\calB(\om)^\bot&&,\\
\clbHt{s}(\curln,\om)&=\clH{s}(\curln,\om)\cap\nH{}{}{\eps}{}(\om)^{\bot_\eps}=\clH{s}(\curln,\om)\cap\calBon{2}{}(\om)^\bot&&,\\
\clbH{s}(\curlonn,\om)&=\clH{s}(\curlonn,\om)\cap\dH{}{}{\eps}{}(\om)^{\bot_\eps}&&,
\end{align*}
and thus except of the last one the definitions of these spaces extend to all $s\in\rz$\,. Moreover, we set for $s>-1/2$
$$\clbLzseps(\om):=\Lzsom\cap\dH{}{}{\eps}{}(\om)^{\bot_\eps}\qqtext{,}\clbLtzseps(\om):=\Lzsom\cap\nH{}{}{\eps}{}(\om)^{\bot_\eps}\qquad.$$
We get

\begin{lem}\mylabel{clalzdirichletzerl}
Let $s>-1/2$\,. Then the direct decompositions
\begin{align*}
\Lzsom&=\clbLzseps(\om)\dotplus\Lin\calBon{1}{}(\om)\qquad,\\
\Lzsom&=\clbLtzseps(\om)\dotplus\eps^\me\Lin\calBon{2}{}(\om)=\clbLtzseps(\om)\dotplus\Lin\calB(\om)
\end{align*}
hold. If additionally $s<1/2$\,, then
$$\Lzsom=\clbLzseps(\om)\oplus_\eps\dH{}{}{\eps}{}(\om)\qqtext{,}\Lzsom=\clbLtzseps(\om)\oplus_\eps\nH{}{}{\eps}{}(\om)\qquad.$$
\end{lem}

Let us note that the operator $\Delta_\eps$ reads as follows:
\begin{table}[h!]
\begin{center}
\fbox{\begin{tabular}{|c||c|}
\hline
$q$ & $\Delta_\eps$ \\
\hline\hline
$0$ & $\eps^\me\pdiv\grad$ \\
\hline
$1$ & $\grad\pdiv-\eps^\me\curl\curl$ \\
\hline
$2$ & $-\curl\curl+\eps^\me\grad\pdiv$ \\
\hline
$3$ & $\pdiv\grad$ \\
\hline
\end{tabular}}
\end{center}
\caption{\label{genlap}Generalized Laplacian}
\end{table}\\
Thus we define
$$\square_\eps:=\grad\pdiv-\eps^\me\curl\curl=\Delta-\epsh\curl\curl\qquad.$$
Always assuming $s\notin\tilde{\pI}$\,, i.e. for all $n\in\nzn$
$$s\neq n+1/2\qqtext{and}s\neq -n-3/2\qquad,$$
we obtain

\begin{theo}\mylabel{cladecotheo}
The following decompositions hold:
\begin{itemize}
\item[\bf(i)] If $s<-3/2$\,, then
\begin{align*}
\Lzsom&=\clH{s}(\curlonn,\om)+\eps^\me\clH{s}(\pdivn,\om)=\clH{s}(\curln,\om)+\eps^\me\clH{s}(\pdivonn,\om)\\
&=\clH{s}(\curlonn,\om)+\eps^\me\clbH{s}(\pdivn,\om)=\clbHt{s}(\curln,\om)+\eps^\me\clbHt{s}(\pdivonn,\om)\quad.
\end{align*}
In the first line the intersections equal the finite dimensional space of Dirichlet resp. Neumann fields
$\dH{}{s}{\eps}{}(\om)$ resp. $\nH{}{s}{\eps}{}(\om)$
and in the second line the intersections equal the finite dimensional space of Dirichlet resp. Neumann fields
$\dH{}{s}{\eps}{}(\om)\cap\calBon{1}{}(\om)^{\bot_\eps}$ resp. $\nH{}{s}{\eps}{}(\om)\cap\calBon{2}{}(\om)^\bot$\,.
\item[\bf(ii)] If $-3/2<s\leq-1/2$\,, then
\begin{align*}
\Lzsom&=\clH{s}(\curlonn,\om)\dotplus\eps^\me\clbH{s}(\pdivn,\om)\\
&=\clbHt{s}(\curln,\om)\dotplus\eps^\me\clbHt{s}(\pdivonn,\om)\dotplus\nH{}{}{\eps}(\om)\qquad.
\end{align*}
\item[\bf(iii)] If $-1/2<s<3/2$\,, then
\begin{align*}
\clbLzseps(\om)&=\clbH{s}(\curlonn,\om)\dotplus\eps^\me\clbH{s}(\pdivn,\om)\qquad,\\
\clbLtzseps(\om)&=\clbHt{s}(\curln,\om)\dotplus\eps^\me\clbHt{s}(\pdivonn,\om)\qquad.
\end{align*}
For $s\geq0$ this decomposition is even $\skp{\eps\,\cdot\,}{\,\cdot\,}_{\lzom}$-orthogonal.
\item[\bf(iv)] If $s>3/2$\,, then
\begin{align*}
\clbLzseps(\om)&=\Big(\big([\Lzsom\boxplus\eta\towerV{-1}{s}]\cap\clbH{<\frac{3}{2}}(\curlonn,\om)\big)\\
&\qquad\qquad\oplus_\eps\eps^\me\big([\Lzsom\boxplus\eta\towerV{-1}{s}]\cap\clbH{<\frac{3}{2}}(\pdivn,\om)\big)\Big)\cap\Lzsom\qquad,\\
\clbLtzseps(\om)&=\Big(\big([\Lzsom\boxplus\eta\towerV{-1}{s}]\cap\clbHt{<\frac{3}{2}}(\curln,\om)\big)\\
&\qquad\qquad\oplus_\eps\eps^\me\big([\Lzsom\boxplus\eta\towerV{-1}{s}]\cap\clbHt{<\frac{3}{2}}(\pdivonn,\om)\big)\Big)\cap\Lzsom
\intertext{and}
\clbLzseps(\om)&=\clbH{s}(\curlonn,\om)\dotplus\eps^\me\clbH{s}(\pdivn,\om)\dotplus\square_\eps\eta\towerP{}{s-2}\qquad,\\
\clbLtzseps(\om)&=\clbHt{s}(\curln,\om)\dotplus\eps^\me\clbHt{s}(\pdivonn,\om)\dotplus\square_\eps\eta\towerP{}{s-2}\qquad,
\end{align*}
where the first two terms in the latter two decomposition are
$\skp{\eps\,\cdot\,}{\,\cdot\,}_{\lzom}$-orthogonal as well. Furthermore,
\begin{align*}
\Lzsom\cap\dH{}{-s}{\eps}{}(\om)^{\bot_\eps}&=\clbH{s}(\curlonn,\om)\oplus_\eps\eps^\me\clbH{s}(\pdivn,\om)\qquad,\\
\Lzsom\cap\nH{}{-s}{\eps}{}(\om)^{\bot_\eps}&=\clbHt{s}(\curln,\om)\oplus_\eps\eps^\me\clbHt{s}(\pdivonn,\om)\qquad.
\end{align*}
\end{itemize}
\end{theo}

\begin{rem}\mylabel{cladecotheorem}
Here Remark \ref{decorem} holds analogously.
In particular the matrix $\eps$ may be moved to the $\curl$-free terms in our decompositions as well and
for $s<-3/2$ and $\tau\geq1/2$ we have
\begin{align*}
\dH{}{s}{\eps}{}(\om)&=\dH{}{}{\eps}{}(\om)\dotplus\dH{}{s}{\eps}{}(\om)\cap\calBon{1}{}(\om)^{\bot_\eps}\qquad,\\
\nH{}{s}{\eps}{}(\om)&=\nH{}{}{\eps}{}(\om)\dotplus\nH{}{s}{\eps}{}(\om)\cap\calBon{2}{}(\om)^\bot\qquad.
\end{align*}
\end{rem}

\begin{theo}\mylabel{cladecogradcurltheosmall}
Let $s\in\rz$\,. Then
\begin{align*}
\text{\bf(i)}&&\clbHt{s}(\curln,\om)&=\grad\clH{s-1}(\grad,\om)
\intertext{and for $s>-1/2$}
\text{\bf(i$'$)}&&\clbH{s}(\curlonn,\om)&=\grad\clH{s-1}(\gradon,\om)\qquad.
\intertext{If $s<5/2$\,, then}
\text{\bf(ii)}&&\clbHt{s}(\pdivonn,\om)&=\curl\clbH{s-1}(\curlon,\pdivn\mu,\om)\\
&&&=\curl\clH{s-1}(\curlon,\pdivn\mu,\om)=\curl\clH{s-1}(\curlon,\om)\qquad,\\
&&\clbH{s}(\pdivn,\om)&=\curl\clbHt{s-1}(\curl,\pdivonn\mu,\om)\\
&&&=\curl\clH{s-1}(\curl,\pdivonn\mu,\om)=\curl\clH{s-1}(\curl,\om)\qquad.
\intertext{All these spaces are closed subspaces of $\Lzsom$\,. For $s<3/2$}
\text{\bf(iii)}&&\Lzsom&=\pdiv\clbHt{s-1}(\pdivon,\curln\mu,\om)\\
&&&=\pdiv\clH{s-1}(\pdivon,\curln\mu,\om)=\pdiv\clH{s-1}(\pdivon,\om)\qquad,\\
&&\Lzsom&=\pdiv\clbH{s-1}(\pdiv,\curlonn\mu,\om)\qqtext{, if $s>1/2$}\qquad,\\
&&\Lzsom&=\pdiv\clH{s-1}(\pdiv,\curlonn\mu,\om)=\pdiv\clH{s-1}(\pdiv,\om)\qquad.
\end{align*}
\end{theo}

\begin{theo}\mylabel{cladecogradcurltheolarge}
\begin{itemize}
\item[\bf(i)] For $s>5/2$
\begin{align*}
\clbHt{s}(\pdivonn,\om)&=\curl\Big(\big(\clH{s-1}(\curlon,\om)\boxplus\eta\towerV{-1}{s-1}\big)\cap\clbH{<\frac{3}{2}}(\pdivn\mu,\om)\Big)\\
&=\curl\big(\clH{s-1}(\curlon,\pdiv\mu,\om)\cap\calBon{1}{}(\om)^{\bot_\mu}\big)=\curl\clH{s-1}(\curlon,\om)\quad,\\
\clbH{s}(\pdivn,\om)&=\curl\Big(\big(\clH{s-1}(\curl,\om)\boxplus\eta\towerV{-1}{s-1}\big)\cap\clbHt{<\frac{3}{2}}(\pdivonn\mu,\om)\Big)\\
&=\curl\big(\clH{s-1}(\curl,\pdivon\mu,\om)\cap\calB(\om)^{\bot_\mu}\big)=\curl\clH{s-1}(\curl,\om)
\end{align*}
are closed subspaces of $\Lzsom$\,.
\item[\bf(ii)] For $s>3/2$
\begin{align*}
&\qquad\big(\Lzsom\boxplus\eta\towerV{-1}{s}\big)\cap\clbH{<\frac{3}{2}}(\curlonn,\om)\\
&=\grad\big(\clH{s-1}(\gradon,\om)\boxplus\eta\towerU{0}{s-1}\big)=\clbH{s}(\curlonn,\om)\dotplus\grad\eta\towerU{0}{s-1}\qquad,\\
&\\
&\qquad\big(\Lzsom\boxplus\eta\towerV{-1}{s}\big)\cap\clbHt{<\frac{3}{2}}(\curln,\om)\\
&=\grad\big(\clH{s-1}(\grad,\om)\boxplus\eta\towerU{0}{s-1}\big)=\clbHt{s}(\curln,\om)\dotplus\grad\eta\towerU{0}{s-1}\qquad,\\
&\\
&\qquad\big(\Lzsom\boxplus\eta\towerV{-1}{s}\big)\cap\clbHt{<\frac{3}{2}}(\pdivonn,\om)\\
&=\curl\Big(\big(\clH{s-1}(\curlon,\om)\boxplus\eta\towerV{-1}{s-1}\boxplus\eta\towerV{0}{s-1}\big)\cap\clbH{<\frac{1}{2}}(\pdivn\mu,\om)\Big)\\
&=\clbHt{s}(\pdivonn,\om)\dotplus\curl\mu^\me\eta\towerV{0}{s-1}\qquad,\\
&\\
&\qquad\big(\Lzsom\boxplus\eta\towerV{-1}{s}\big)\cap\clbH{<\frac{3}{2}}(\pdivn,\om)\\
&=\curl\Big(\big(\clH{s-1}(\curl,\om)\boxplus\eta\towerV{-1}{s-1}\boxplus\eta\towerV{0}{s-1}\big)\cap\clbHt{<\frac{1}{2}}(\pdivonn\mu,\om)\Big)\\
&=\clbH{s}(\pdivn,\om)\dotplus\curl\mu^\me\eta\towerV{0}{s-1}
\end{align*}
are closed subspaces of $\Lzsom\boxplus\eta\towerV{-1}{s}$\,. Moreover, $\pdiv\eta V^0_{-,n,m}=0$\,.
\end{itemize}
\end{theo}

\begin{theo}\mylabel{cladecogradcurltheoexcept}
Let $s>3/2$\,. Then
\begin{align*}
\text{\bf(i)}&&\Lzsom&=\pdiv\Big(\big(\clH{s-1}(\pdivon,\om)\boxplus\eta\towerV{-1}{s-1}\boxplus\eta\towerUh{-1}{}\big)\cap\clbHt{<\frac{1}{2}}(\curln\mu,\om)\Big)\\
&&&=\pdiv\Big(\big(\clH{s-1}(\pdivon,\curl\mu,\om)\cap\calBon{2}{}(\om)^{\bot_\mu}\big)\boxplus\eta\towerUh{-1}{}\Big)\\
&&&=\pdiv\big(\clH{s-1}(\pdivon,\curl\mu,\om)\cap\calBon{2}{}(\om)^{\bot_\mu}\big)\dotplus\Delta\eta\towerUh{0}{}\\
&&&=\pdiv\clH{s-1}(\pdivon,\om)\dotplus\Delta\eta\towerUh{0}{}\qquad,\\
\text{\bf(ii)}&&\Lzsom&=\pdiv\Big(\big(\clH{s-1}(\pdiv,\om)\boxplus\eta\towerV{-1}{s-1}\boxplus\eta\towerUh{-1}{}\big)\cap\clbH{<\frac{1}{2}}(\curlonn\mu,\om)\Big)\\
&&&=\pdiv\big(\clH{s-1}(\pdiv,\curlon\mu,\om)\boxplus\eta\towerUh{-1}{}\big)\\
&&&=\pdiv\clH{s-1}(\pdiv,\curlon\mu,\om)\dotplus\Delta\eta\towerUh{0}{}\\
&&&=\pdiv\clH{s-1}(\pdiv,\om)\dotplus\Delta\eta\towerUh{0}{}\qquad.
\end{align*}
\end{theo}

\begin{acknow}
The author is particularly indebted to his colleagues and friends Sebastian Bauer and Michael Trebing for many fruitful discussions.
\end{acknow}

\end{document}